%% file: ex_article.tex
\documentclass[]{siamart171218}

\def\Rd{} 
\def\B{\color{black}}

\input{ex_shared}
\input{new_commands}

\ifpdf
\hypersetup{
 pdftitle={A Tchebycheffian extension of multi-degree B-splines: Algorithmic computation and properties},
 pdfauthor={R. R. Hiemstra, T. J. R. Hughes, C. Manni, H. Speleers, and D. Toshniwal}
}
\fi

\begin{document}

\maketitle

\begin{abstract}
In this paper we present an efficient and robust approach to compute a normalized B-spline-like basis for spline spaces with pieces drawn from extended \Chefshort spaces. The extended \Chefshort spaces and their dimensions are allowed to change from interval to interval. The approach works by constructing a matrix that maps a generalized Bernstein-like basis to the B-spline-like basis of interest. The B-spline-like basis shares many characterizing properties with classical univariate B-splines and may easily be incorporated in existing spline codes. This may contribute to the full exploitation of \Chef splines in applications, freeing them from the restricted role of an elegant theoretical extension of polynomial splines.  Numerical examples are provided that illustrate the procedure described. 
\end{abstract}



\input{sections/introduction}

\input{sections/preliminaries_c}

\input{sections/section3_c}

\input{sections/section3_rec}

\input{sections/section4}

\input{sections/section5}

\input{sections/section6_gb}

\input{sections/section6_c}

\input{sections/conclusion}

\section*{Acknowledgments}
R. R. Hiemstra, T. J. R. Hughes, and D. Toshniwal were partially supported by the Office of Naval Research (N00014-17-1-2119, N00014-17-1-2039, and N00014-13-1-0500), by the Army Research Office (W911NF-13-1-0220), by the National Institutes of Health (5R01HL129077-02,03), and by the National Science Foundation Industry/University Cooperative Research Center (IUCRC) for Efficient Vehicles and Sustainable Transportation Systems (EV-STS), and the United States Army CCDC Ground Vehicle Systems Center (TARDEC/NSF Project 1650483 AMD 2).  C. Manni and H. Speleers were partially supported by the Mission Sustainability Programme of the University of Rome Tor Vergata through the project IDEAS (CUP E81I18000060005) and by the MIUR Excellence Department Project awarded to the Department of Mathematics, University of Rome Tor Vergata (CUP E83C18000100006); they are members of Gruppo Nazionale per il Calcolo Scientifico -- Istituto Nazionale di Alta Matematica. This support is gratefully acknowledged.  

\bibliographystyle{siamplain}
\bibliography{sections/bibliography_c}
\end{document}

%% file: ex_shared.tex
\usepackage{lipsum}
\usepackage{amsfonts}
\usepackage{graphicx}
\usepackage{epstopdf}
 \usepackage{mathtools}
\usepackage{algpseudocode}
\usepackage{subfigure}

\ifpdf
  \DeclareGraphicsExtensions{.eps,.pdf,.png,.jpg}
\else
  \DeclareGraphicsExtensions{.eps}
\fi

\newsiamremark{example}{Example}
\newsiamremark{remark}{Remark}

\headers{A Tchebycheffian extension of multi-degree B-splines}{R. R. Hiemstra, T. J. R. Hughes, C. Manni, H. Speleers, and D. Toshniwal}

\title{A Tchebycheffian extension of multi-degree B-splines: Algorithmic computation and properties\thanks{Submitted to the editors DATE.}}

\author{Ren\'e R. Hiemstra\footnotemark[3]~\thanks{Institut f{\"u}r Baumechanik und Numerische Mechanik, Leibniz Universit{\"a}t Hannover, Germany (\email{rene.hiemstra@ibnm.uni-hannover.de})}
\and Thomas J. R. Hughes\thanks{Oden Institute for Computational Engineering and Sciences, University of Texas at Austin, USA (\email{hughes@oden.utexas.edu}).}
\and Carla Manni\thanks{Department of Mathematics, University of Rome Tor Vergata, Italy (\email{manni@mat.uniroma2.it}, \email{speleers@mat.uniroma2.it}).}
\and Hendrik Speleers\footnotemark[4]
\and Deepesh Toshniwal\footnotemark[3]~\thanks{Delft Institute of Applied Mathematics, Delft University of Technology, The Netherlands (\email{d.toshniwal@tudelft.nl})}}

\usepackage{amsopn}


%% file: new_commands.tex
\newcommand{\domain}{\Delta}

\newcommand{\interval}{J}

\renewcommand{\a}{\ensuremath{a}}			
\renewcommand{\b}{\ensuremath{b}}			
\renewcommand{\aa}{\ensuremath{z}}	        
\newcommand{\I}{\ensuremath{J}}			

\newcommand{\lmult}{\ensuremath{l}}		
\newcommand{\rmult}{\ensuremath{l}}		
\newcommand{\lintsum}{\ensuremath{\mu_{\mbf{\rknot}}}}		
\newcommand{\rintsum}{\ensuremath{\mu_{\mbf{\lknot}}}}		
\newcommand{\lsmooth}{\ensuremath{r_{\mbf{\lknot}}}}
\newcommand{\rsmooth}{\ensuremath{r_{\mbf{\rknot}}}}
\newcommand{\lsmoothtilde}{\ensuremath{r_{\tilde{\mbf{\lknot}}}}}
\newcommand{\rsmoothtilde}{\ensuremath{r_{\tilde{\mbf{\rknot}}}}}

\newcommand{\degree}{p}

\newcommand{\bdegree}{\mbf{\degree}}
\newcommand{\smooth}{r}

\newcommand{\bsmooth}{\mbf{\smooth}}

\newcommand{\splSpace}{\mathbb{S}}
\newcommand{\splSpacep}{\splSpace^{\bdegree}(\domain)}
\newcommand{\splSpacerp}{\splSpace_\bsmooth^\bdegree(\domain)}
\newcommand{\splSpacerpDer}{\widehat{\splSpace_\bsmooth^\bdegree}(\domain)}
\newcommand{\splSpacerpA}{\splSpace_\bsmooth^\bdegree(\domain)}
\newcommand{\splSpacerpB}{\splSpace_{\tilde{\bsmooth}}^{\bdegree}(\domain)}

\newcommand{\bspint}{d}
\newcommand{\bbint}{b}

\newcommand{\Span}[1]{\ensuremath{\mathrm{span}\left\{#1\right\}}}

\newcommand{\N}{\ensuremath{n}}              	
\newcommand{\M}{\ensuremath{\theta}}              	
\renewcommand{\O}{\ensuremath{\phi}}             

\renewcommand{\SS}[2]{\ensuremath{\mathbb{S}_{#1}^{#2}}}				
\newcommand{\GP}[2]{\ensuremath{\mathbb{G}_{#1}^{#2}}}
\newcommand{\ECT}[2]{\ensuremath{\mathbb{T}_{#1}^{#2}}}	
\newcommand{\ECTtilde}[2]{\ensuremath{\widehat{\mathbb{T}}_{#1}^{#2}}}			
\newcommand{\uu}{\ensuremath{g}}									
\newcommand{\bs}{\ensuremath{B}}									
\newcommand{\BS}{\ensuremath{N}}	
\newcommand{\MS}{\ensuremath{M}}								
\newcommand{\BStilde}{\widehat{\BS}}
\newcommand{\MStilde}{\widehat{\MS}}

\newcommand{\jump}[2]{\ensuremath{\mathrm{Jump}_{#1} (#2)}}	
\renewcommand{\r}{\ensuremath{r}}									
\newcommand{\R}{\mbf{r}}										

\newcommand{\nelms}{\ensuremath{m}}							

\newcommand{\vect}[2]{\ensuremath{\boldsymbol{\mathrm{#1}}_{#2}}}
\newcommand{\mat}[1]{\ensuremath{\mathsf{#1}}}
\newcommand{\List}[1]{\ensuremath{\left\{ \right. #1 \left.\right\} }}

\newcommand{\dd}{\ensuremath{\,\mathrm{d}}}
\newcommand{\x}{\ensuremath{x}}

\newcommand{\y}{\ensuremath{y}}

\newcommand{\p}{\ensuremath{p}}
\newcommand{\q}{\ensuremath{q}}
\renewcommand{\P}{\mbf{\p}}

\newcommand{\knot}{\ensuremath{\xi}}
\newcommand{\lknot}{\ensuremath{u}}
\newcommand{\rknot}{\ensuremath{v}}
\newcommand{\normalize}[1]{\ensuremath{\widetilde{#1}}}

\newcommand{\dof}{\ensuremath{d}}
\newcommand{\newdof}{\ensuremath{\tilde{d}}}
\newcommand{\C}{\ensuremath{\mat{\tilde{C}}}}
\newcommand{\CC}{\ensuremath{\mat{C}}}
\newcommand{\ca}{\ensuremath{\alpha}}
\newcommand{\cb}{\ensuremath{\beta}}
\newcommand{\A}[2]{\ensuremath{A_{#1}(#2)}}
\renewcommand{\AA}{\ensuremath{\mat{A}}}
\newcommand{\cons}{\ensuremath{\rho}}

\newcommand{\spR}{\ensuremath{\mathbb{R}}}

\newcommand{\highlight}[1]{{\sl #1}}

\newcommand{\Chef}{Tchebycheffian }
\newcommand{\Chefshort}{Tchebycheff }
\newcommand{\Bezier}{B\'ezier~}

\newcommand*\Let[2]{\State #1 $\gets$ #2}

\newcommand{\suppA}[2]{[\lknot_{#1}, \rknot_{#2}]}
\newcommand{\suppB}[2]{[\tilde{\lknot}_{#1}, \tilde{\rknot}_{#2}]}
\newcommand{\tripleA}[2]{\left( \suppA{#1}{#2}, \lsmooth(#1), \rsmooth(#2) \right)}
\newcommand{\tripleB}[2]{\left( \suppB{#1}{#2}, \lsmoothtilde(#1), \rsmoothtilde(#2) \right)}

\newcommand{\mbf}[1]{{\boldsymbol{#1}}}
\newcommand{\RR}{{\mathbb{R}}}

\newcommand{\ZZ}{{\mathbb{Z}}}

\newcommand{\dint}{{\mathrm{d}}}

%% file: sections/introduction.tex
\section{Introduction} \label{sec:introduction}
In the classical polynomial setting, univariate multi-degree splines are piecewise polynomial functions that are glued together in a certain smooth way and where the various pieces can have different degrees \cite{Beccari:2017,ShenW:2010a}. This multi-degree formulation offers significant advantages with respect to the classical uniform-degree case, allowing for the modeling of complex geometries with fewer control points and more versatile adaptive schemes in numerical simulation \cite{Sederberg:2003a,Toshniwal:2017a}.

Polynomial splines, in both the uniform-degree or multi-degree version, can be seen as a special case of \Chef splines \cite{Buchwald:2003,Mazure:2011a,Nurnberger:1984,Schumaker:2007}, i.e., smooth piecewise functions whose pieces are drawn from extended \Chefshort spaces (ET-spaces). ET-spaces are natural generalizations of algebraic polynomial spaces \cite{KarlinS:1966,Schumaker:2007} because they satisfy the same bounds on the number of zeros of non-trivial elements. Relevant examples of ET-spaces are nullspaces of linear differential operators on suitable intervals \cite{Coppel:1971,Schumaker:2007}. \Chef splines share many properties with the classical polynomial splines but also offer a much more flexible framework, due to the wide diversity of ET-spaces. Multivariate extensions of \Chef splines can be easily obtained via (local) tensor-product structures \cite{BraccoLMRS:2016-gb,BraccoLMRS:2016,BraccoLMRS:2019}.

The rich variety of parameters in {\Rd \Chef spline spaces (and ET-spaces)} has been explored in free-form design and constrained interpolation/approximation; see \cite{Carnicer:2003, Costantini:2005, LycheS:2000,Mazure:2018, Schweikert:1966, WangF:2008} and references therein. In addition, \Chef splines emerge as a natural tool in several engineering contexts. Among others, \Chef splines based on trigonometric and/or exponential functions allow for an exact representation of conic sections with (almost) arc-length parameterization, without the need for a rational form \cite{Mainar:2001}. As a consequence, their elegant behavior with respect to  differentiation and integration makes them an appealing substitute for the rational NURBS model in the framework of both  Galerkin and collocation isogeometric methods \cite{Aimi:2017,ManniPS:2011a,ManniRS:2015,ManniRS:2017}. When the geometry is not an issue, \Chef splines can still provide an interesting problem-dependent alternative to classical polynomial B-splines/NURBS for solving differential problems: they allow for an efficient treatment of sharp gradients and thin layers \cite{ManniPS:2011b,ManniRS:2015} and are able to outperform classical polynomial B-splines in the spectral approximation of differential operators \cite{ManniRS:2015,ManniRS:2017}.

The success of polynomial splines greatly relies on the famous B-spline basis which can also be defined in the multi-degree setting \cite{Beccari:2017,Nurnberger:1984,ShenW:2010a,ShenW:2010b,Toshniwal:2017a}. Most of the results known for polynomial splines extend in a natural way to \Chef splines. However, the possibility of representing the space in terms of a basis with similar properties to polynomial B-splines is not always guaranteed,  even  for pieces taken from ET-spaces of the same dimension. More precisely, there are two main categories of \Chef splines: the various pieces are drawn either from the same ET-space --- see \cite{Nurnberger:1984} for a proper meaning in case of different local dimensions --- or from different ET-spaces. In the former case, \Chef splines always admit a representation in terms of basis functions with similar properties to polynomial B-splines. The latter offers a much more general framework --- sometimes referred to in the literature as piecewise \Chef splines \cite{Mazure:2018} --- and allows us to optimally benefit from the great diversity of ET-spaces, but the existence of a B-spline-like basis requires constraints on the various ET-spaces.

\Chef splines can be easily incorporated in existing spline codes because the corresponding B-spline-like basis, whenever it exists, is  compatible with classical B-splines  as it enjoys the same structural properties. When all the ET-spaces have the same dimension, various approaches have been used in the \Chef setting to construct such a B-spline-like basis: generalized divided differences \cite{Muhlbach:2006,Schumaker:2007}, Hermite interpolation \cite{Buchwald:2003,Nurnberger:1984}, integral recurrence relations \cite{BisterP:1997,Lyche:2019}, de Boor-like recurrence relations \cite{DynR:1988,Lyche:1985}, and blossoming \cite{Mazure:2011a}. Each of these definitions has advantages according to the problem one has to face or to the properties to be proved. All these constructions lead to the same functions, up to a proper scaling. For non-uniform local dimensions, the literature is much less developed and B-spline-like bases have been constructed via Hermite interpolation \cite{Buchwald:2003,Nurnberger:1984}.

Unfortunately, none of the currently available constructions for B-spline-like bases of \Chef spline spaces is very well suited for their efficient and robust numerical evaluation and manipulation, due to computational complexity and/or numerical instabilities. This drawback has seriously penalized \Chef splines, so far, in practical applications despite their great potential, and has confined them mostly to the role of an elegant theoretical extension of the polynomial case.



\Rd
The aim of this paper is to formulate an approach that circumvents the aforementioned complexity in working with \Chef splines. We do so by focusing on spaces of \Chef splines with pieces drawn from different ET-spaces of possibly different dimensions. Following \cite{Nurnberger:1984}, we refer to these spaces as \highlight{generalized \Chef splines (GT-splines)}. The corresponding B-spline-like basis, in case it exists, will be referred to  as \highlight{generalized \Chef B-splines (GTB-splines)}. Then, our main contribution is an efficient and robust algorithm for evaluation of GTB-splines, whenever they exist. \B

The algorithm proceeds by incrementally increasing the smoothness at the breakpoints starting from the space of piecewise discontinuous functions obtained by collecting the various ET-spaces which are represented in terms of a Bernstein-like basis. At each step, the smoothness constraints are represented in the form of a matrix whose nullspace identifies the basis elements. The algorithm explicitly constructs this nullspace without solving any linear systems. In other words, at each step, the algorithm explicitly constructs a matrix that specifies how GTB-splines that are $C^r$ at some breakpoint can be linearly combined to form GTB-splines that are $C^{r+1}$. The product of all such matrices is called an \highlight{extraction operator} and it maps the Bernstein-like basis to GTB-splines. In fact, we prove that the output of the algorithm is exactly the entire set of GTB-splines that span the considered GT-spline space. The algorithm can be seen as a \Chef extension of the one proposed in \cite{Speleers:2018,Toshniwal:2017a,Toshniwal:2018b} for multi-degree polynomial splines. In order to ensure existence of a GTB-spline basis, we consider the sufficient conditions proposed in \cite{Buchwald:2003} which can be easily checked and are satisfied for a wide class of GT-splines of interest in applications. However, this is not a limitation of the algorithm we are proposing; the algorithm produces the required GTB-spline basis whenever it exists.

\Rd
The above contribution\B~is complemented by additional results: we provide a knot insertion formula and a global integral recurrence relation for GTB-splines. While the former is in complete analogy with the one known for the multi-degree polynomial case \cite{Beccari:2017,Toshniwal:2018b}, the latter is a new contribution also for the multi-degree polynomial case, where only local integral recurrence relations have been proposed so far in the literature \cite{Beccari:2017, ShenW:2010a}. The provided global integral recurrence relation completely mimics the one known for polynomial/\Chef splines of uniform degree/local dimension and is expressed in an elegant way by using an extension of the concept of weights.

The remainder of the paper is organized as follows. Section~\ref{sec:preliminaries} recalls  several  properties of ET-spaces, introduces notation, and defines the space of GT-splines. The existence, under proper assumptions, of GTB-splines is summarized in Section~\ref{sec:3}; it basically collects in a homogeneous and self-contained presentation results from \cite{Buchwald:2003,Nurnberger:1984}. Section~\ref{sec:3_rec} presents local and global integral recurrence relations for GTB-splines, while Section~\ref{sec:4} is devoted to the knot insertion formula which is the main ingredient for the evaluation algorithm described and analyzed in Section~\ref{sec:5}.
An interesting case study is detailed in Section~\ref{sec:6_gb}, and some numerical examples are collected in Section~ \ref{sec:6}. We end with some concluding remarks in Section~\ref{sec:conclusion}.

%% file: sections/preliminaries_c.tex
\section{Preliminaries} \label{sec:preliminaries}
We are interested in piecewise functions, whose pieces belong to ET-spaces and are glued together in a certain smooth way. We first define ET-spaces on a real interval $\interval$ (see, e.g., \cite{Schumaker:2007}).

\begin{definition}[Extended \Chefshort space]
\label{def:ET}
Given 
an interval $\interval$, a space $\ECT{\p}{}(\interval)\subset C^{\p}(\interval)$ of dimension $\p+1$ is an \highlight{extended \Chefshort \mbox{(ET-)} space} on $\interval$ if any Hermite interpolation problem with $\p+1$ data on $\interval$ has a unique solution in $\ECT{\p}{}(\interval)$.
In other words, for any integer $m\geq1$,
let ${\bar x}_1,\ldots, {\bar x}_m$ be distinct points in $\interval$ and
let $d_1,\ldots,d_m$ be nonnegative integers such that $\p+1=\sum_{i=1}^m (d_i+1)$. Then,
for any set $\{f_{i,j}\in\RR\}_{i=1,\ldots,m,\, j=0,\ldots,d_i}$
there exists a unique $q\in\ECT{\p}{}(\interval)$ such that
\begin{equation*}
D_{}^j q({\bar x}_i)=f_{i,j}, \quad i=1,\ldots,m, \quad j=0,\ldots,d_i.
\end{equation*}
\end{definition}

If $\interval$ is a bounded closed interval, then any ET-space of dimension $\p+1$ on $\interval$ is an extended complete Tchebycheff (ECT-) space on $\interval$, i.e., it is spanned by the following functions (see \cite{Mazure:2007, Schumaker:2007}):
\begin{align}
\label{eq:gen-powers}
\begin{cases}
	\uu_0(\x) := w_0(\x), \\
	\uu_1(\x) := w_0(\x) \int_{\aa}^\x w_1(\y_1) \dint\y_1, \\
	\hspace*{1.1cm} \vdots \\
	\uu_\p(\x) := w_0(\x) \int_{\aa}^\x w_1(\y_1) \int_{\aa}^{\y_1} \cdots   \int_{\aa}^{y_{\Rd \p-1 \B}} w_\p(\y_\p) \dint\y_{\p} \cdots \dint\y_1,
	\end{cases}
\end{align}
where $\aa$ is any point in $\interval$ and $w_{j} \in C^{\p-j}(\interval)$, $j=0,1,\ldots , \p$ are positive functions called \highlight{weights}. The functions $\Rd \uu_0, \ldots, \uu_\p$ are called \highlight{generalized powers}.

\begin{remark}\label{rmk:weights_factor}
A given ECT-space can be identified by different sets of weights; see \cite{Lyche:2019} for details and examples. In particular, it is clear that
the two weight systems
\begin{equation*}
w_0,\ldots,w_p \quad {\rm and} \quad K_0w_0,\ldots,K_pw_p,
\end{equation*}
where $K_0,\ldots, K_p$ are positive constants, identify the same ECT-space.
\end{remark}
\begin{remark}\label{rmk:gen-powers-der}
A very important case for applications is $w_0=1$ with $\p\geq 1$. Under such assumptions, it can be directly checked that
\begin{align}\label{eq:gen-powers-der}
\begin{cases}
	D \uu_0(\x) = 0, \\
	D \uu_1(\x) = w_1(\x), \\
	\hspace*{1.35cm}\vdots \\
	D \uu_\p(\x) = w_1(\x) \int_{\aa}^{\x} \Rd w_2(\y_2) \int_{\aa}^{\y_{2}} \B \cdots   \int_{\aa}^{\y_{\Rd \p-1 \B}} w_\p(\y_\p) \dint\y_{\p} \cdots \dint\y_2,
	\end{cases}
\end{align}
i.e., the space spanned by the derivatives of the functions in \cref{eq:gen-powers} is an ECT-space of dimension $\p$ on $\interval$ and it is identified by the weights $w_1, \ldots, w_\p$.
\end{remark}
\begin{remark}
The polynomial space of degree $p$ fits in this framework by taking $w_0=\cdots=w_p=1$. In this case, the functions in \cref{eq:gen-powers} become
$\uu_j={(x-\aa)^j}/{(j!)}$,
which are the standard (polynomial) power functions.
\end{remark}

Pieces of our splines shall be drawn from arbitrary ECT-spaces of possibly different dimensions. Consider a partitioning, $\domain$, of the interval $[\a,\b] \subset \spR$ into a sequence of breakpoints,
\begin{equation*}
	\domain := \List{\a =: \x_0 < \x_1 <  \cdots < \x_{\nelms-1} < \x_{\nelms} := \b}.
\end{equation*}
Furthermore, we set $\I_i := [\x_{i-1}, \x_{i})$, $i=1, \ldots, \nelms-1$, and $\I_\nelms := [\x_{\nelms-1}, \x_{\nelms}]$. We also define an ECT-space of dimension $\p_i+1$ on each closed interval $[\x_{i-1}, \x_{i}]$, $i=1, \ldots, \nelms$:
\begin{equation*}
	\ECT{\p_i}{(i)} := \Span{\uu^{(i)}_0, \ldots , \uu^{(i)}_{\p_i}}, \quad \uu^{(i)}_j \in C^{\p_i}([\x_{i-1}, \x_{i}]), \quad j=0, \ldots, \p_i,
\end{equation*}
where $\uu^{(i)}_0, \ldots , \uu^{(i)}_{\p_i}$ are generalized powers defined in terms of positive weights $w_j^{(i)}\in C^{\p_i-j}([\x_{i-1}, \x_{i}])$, $j=0 \ldots,\p_i$ as in \cref{eq:gen-powers}.
Collectively, these functions span the following space:
\begin{equation*}
\splSpacep := \Big\{ {\Rd s: [\a, \b] \rightarrow \RR :}
	 \left. s \right|_{\I_i} \in \ECT{\p_i}{(i)},\; i=1,\ldots, \nelms \Big\}.
\end{equation*}

In order to measure smoothness at the breakpoints we define the following jump operator for a given $s \in \splSpacep$,
\begin{equation*}
	\jump{\x_i, k}{s} := D^{k}_{-}s(\x_i) - D^{k}_{+}s(\x_i).
\end{equation*}
Then, 
we can define the space of generalized \Chef splines as follows.

\begin{definition}[Generalized \Chef splines] \label{def:spline-space}
Given the sets of integers $\P := \{\p_1, \ldots, \p_\nelms\}$ and
\begin{equation}\label{eq:smoothness}
\R := \{\r_i \in \ZZ  : -1 \leq \r_i \leq \min\{\p_{i}, \p_{i+1}\}, \; i=1,\ldots , \nelms-1, \; \r_0 = \r_{\nelms} = -1\},
\end{equation}
 we define
\begin{equation} \label{eq:spline-space}
    \splSpacerp := \Big\{  s \in \splSpacep :  \jump{\x_i, j}{s} = 0, \; j = 0,\ldots, \r_i \text{ and } i = 1, \ldots, \nelms-1   \Big\}.
\end{equation}
\end{definition}

The value $\r_i$ represents the smoothness at breakpoint $\x_i$, $i=1,\ldots,m-1$. The set of linear constraints, encoded in $\jump{\x_i, j}{s} = 0 $, enforce the prescribed smoothness in between adjoining elements of the partition. This leads to a system 
of $\O := \sum_{i=1}^{\nelms-1} (\r_i+1) $ equations in $\M := \sum_{i=1}^{\nelms}(\p_i+1)$ unknowns. All smoothness conditions are linearly independent because the functions $\{u^{(i)}_0(\x), \ldots , u^{(i)}_{\p_i}(\x)\}$ on each interval $\I_i$ form an ECT-system \cite{Buchwald:2003,Nurnberger:1984}. This leads to the following dimension result; see also \cite[Theorem~1.1]{Buchwald:2003}.

\Rd
\begin{proposition}
The dimension of $\splSpacerp$ is
\begin{equation*}
\N := \M-\O=\p_1+1+\sum_{i=1}^{\nelms-1} {(\p_{i+1}-\r_i)}=\p_{\nelms}+1+\sum_{i=1}^{\nelms-1}(\p_i-\r_i).
\end{equation*}
\end{proposition}
\B

Whenever we deal with a single weight system $w_0,\ldots, w_\p$, $\p:=\max_{1\leq i \leq\nelms} \p_i$, defined on the entire interval $[\a, \b]$, all the pieces of the spline functions are basically taken from the ``same'' ECT-space, possibly allowing different local dimensions.
In this case, the spline space in \cref{eq:spline-space} is quite well understood and it enjoys all the nice properties of standard polynomial splines; see \cite{Lyche:2018,Mazure:2011a,Schumaker:2007} and references therein for the case where all the local spaces have the same dimension, and \cite{Nurnberger:1984} for non-uniform dimensions.
On the other hand, in order to fully exploit the richness and the variety of ECT-spaces, it is of interest to consider different ECT-spaces on different intervals. In this much more general framework, obtaining spline spaces equipped with the same properties as standard polynomial splines, including a B-spline-like basis, entails constraints on the various ECT-spaces which can be described in terms of reciprocal smoothness of the associated weight systems.
 From  this perspective, we consider the following definition, which is equivalent to the requirement on the weights in \cite[Lemma~2.7]{Buchwald:2003} taking into account \cref{rmk:weights_factor}.

\begin{definition}[Admissible weights]
\label{weight:assumption}
The weight systems $\{ w_j^{(i)}, \; j=0,\ldots, \p_i \}$ generating the ECT-spaces $ \ECT{\p_i}{(i)}$, $i=1, \ldots, m$, are \highlight{admissible} for the space $\splSpacerp$ if
for $i=1,\ldots,\nelms-1$ and $j=0,\ldots,\r_i$ we have
\begin{equation*}
D_{-}^lw_j^{(i)}(x_i)=D_{+}^lw_j^{(i+1)}(x_i), \quad l=0,\ldots, \r_i-j.
\end{equation*}
\end{definition}


\begin{remark}
How to construct the weights is well known for a single ECT-space \cite{Karlin:1968,Mazure:2011a} but it can be an issue whenever different ECT-spaces are considered \cite{Mazure:2018}. However, there are cases where admissible weights, according to \cref{weight:assumption}, can be easily constructed. For example, they can be obviously extracted from a single weight system such that all the pieces are drawn from the same ECT-space. Furthermore, they can be easily deduced for an interesting class of \Chef splines which allows for the use of different ECT-spaces, the so-called \highlight{generalized polynomial splines}; see \cref{sec:6_gb}. 
For handling more general settings, one could apply the constructive procedure for finding all weight systems associated with a given ECT-space in a bounded closed interval presented in \cite{Mazure:2011b}.
\end{remark}

\begin{remark}
Dealing with admissible weights gives only a sufficient condition for obtaining \Chef splines equipped with a B-spline-like basis; see \cite{Buchwald:2003} and also the next section. The simplicity of this condition and the fact that it embraces relevant classes of \Chef splines motivate our choice. We refer the reader to \cite{Mazure:2018} for explicit necessary and sufficient conditions for smoothly gluing together ECT-spaces of dimension $5$.
\end{remark}

In the next section we summarize from the literature some properties of GT-spline spaces. In particular, we show that, under certain assumptions on the weights, a GT-spline space admits a B-spline-like basis. 

%% file: sections/section3_c.tex
\section{Generalized \Chef B-splines} \label{sec:3}
In this section we introduce basis functions for the GT-spline space $\splSpacerp$ that possess all the characterizing properties of standard polynomial B-splines. B-spline-like bases can be defined under different normalizations. Here, in this paper, we mainly focus on the partition-of-unity normalization, and we call the corresponding functions \highlight{Generalized \Chef B-splines (GTB-splines)}.
The material we are going to present in the current section is greatly inspired by the results in \cite{Buchwald:2003,Nurnberger:1984}.
We provide a concise summary, aiming for a self-contained presentation and a unified notation tailored for the subsequent sections. 
Furthermore, we detail the proofs of results which are not explicitly presented in the literature in the form we need.

We begin by introducing some notation and several results that assist in characterizing GTB-splines.
Similar to polynomial B-splines, GTB-splines can be defined using certain knot vectors. To allow for ECT-spaces of varying dimension, it is convenient to consider two knot vectors,
\begin{subequations}
\begin{align}
\mbf{\lknot}{} &:= (\lknot_k )_{k=1}^{\N}:= (\;
    \underbrace{\x_0,\; \ldots,\;  \x_0}_{\p_{1} - \r_0 \text{ times}},\;  \ldots,\;
    \underbrace{\x_{i},\;  \ldots,\;  \x_{i}}_{\p_{i+1} - \r_i \text{ times}},\;   \ldots,\;
    \underbrace{\x_{\nelms-1},\;  \ldots,\;  \x_{\nelms-1}}_{\p_{\nelms} - \r_{\nelms-1} \text{ times}}
    \; ), 	\label{eq:knots1} \\
\mbf{\rknot}{} &:= (\rknot_k )_{k=1}^{\N}:= (\;
    \underbrace{\x_1,\;  \ldots,\;  \x_1}_{\p_1 - \r_1 \text{ times}},\; \ldots,\;
    \underbrace{\x_{i},\;  \ldots,\;  \x_{i}}_{\p_i - \r_i \text{ times}},\; \ldots ,\;
    \underbrace{\x_\nelms,\;  \ldots,\;  \x_\nelms}_{\p_{\nelms} - \r_{\nelms} \text{ times}}
    \; ). \label{eq:knots2}
\end{align}
\end{subequations}

\begin{remark}
With respect to standard polynomial splines of uniform degree $\p$, the above vectors are related to the standard B-spline knot vector, $\{\knot_k, \; k=1, \ldots , \N+\p \Rd+ 1 \B \}$, via the relationship, $\lknot_k = \knot_k$ and $\rknot_{k} = \knot_{k+\p+1}$, $k = 1, \ldots, \N$.
\end{remark}

\begin{remark}
The use of the two knot vectors \cref{eq:knots1,eq:knots2} greatly simplifies the presentation of the properties of GTB-splines. These two knot vectors have been introduced in \cite{Buchwald:2003} to characterize properly posed Hermite interpolation problems in GT-spline spaces and have been recently exploited to describe B-spline-like bases for multi-degree polynomial splines in \cite{Beccari:2017,Toshniwal:2018b}.
\end{remark}

Mimicking the polynomial spline setting, each of the intervals
\begin{equation*}
[\lknot_k, \rknot_{k}], \quad k=1, \ldots , \N,
\end{equation*}
corresponds to the support of a B-spline-like basis function in $\splSpacerp$.
\cref{lemma:6} implies that these intervals are non-empty and satisfy $[\lknot_k, \rknot_{k}] \cup [\lknot_{k+1}, \rknot_{k+1}] = [\lknot_k, \rknot_{k+1}]$ and $[\lknot_k, \rknot_{k}] \cap [\lknot_{k+1}, \rknot_{k+1}] = [\lknot_{k+1}, \rknot_{k}]$. \cref{lemma:7} shows that there are $\p_i+1$ of such intervals that intersect with element $[\x_{i-1},\x_i)$. Before proving these lemmas, we define two types of quantities,
\begin{equation*}
\rintsum(i) := \sum_{j=0}^{i-1} (\p_{j+1} - \r_{j}), \quad
\lintsum(i) := \sum_{j=1}^{i} (\p_{j} - \r_j), \quad 
i=0,\ldots,\nelms,
\end{equation*}
where an empty sum is assumed to be zero.

\begin{lemma}\label{lemma:6}
Let $k \in  \{1, \ldots, \N\}$ be arbitrary. It holds that $\lknot_{k} \leq \rknot_{k-1}$ and $\lknot_{k} < \rknot_{k}$.
Additionally, if $\r_i \geq 0$ for all $i \in \{1, \dots, m-1\}$, then $\lknot_{k} < \rknot_{k-1} \leq \rknot_{k}$.
\end{lemma}
\begin{proof}
Let $k = \rintsum(i) + \rmult$, where $1 \leq \rmult \leq \p_{i+1}-\r_i$. By inspection, $\lknot_{k} = \x_{i}$. On the other hand,
\begin{align*}
k &= \rintsum(i) +  \rmult = \sum_{j=0}^{i-1}(\p_{j+1}-\r_j) +  \rmult = \sum_{j=1}^{i}(\p_{j}-\r_{j}) + \sum_{j=0}^{i-1}(\r_{j+1}-\r_j) +  \rmult \\
&= \lintsum(i) + \r_i - \r_0 +  \rmult
\geq \begin{dcases}
  \lintsum(i) + 2, & \r_i \geq 0,\\
  \lintsum(i) + 1, & \r_i \geq -1,
\end{dcases}
\end{align*}
since $\r_0 = -1$ and $\rmult \geq 1$.
By inspection, $\rknot_{k} \geq \x_{i+1}$.
Similarly,
\begin{equation*}
\rknot_{k-1} \geq
\begin{dcases}
  \x_{i+1}, & \r_i \geq 0,\\
  \x_{i}, & \r_i \geq -1.
\end{dcases}
\end{equation*}
The above proves the lemma since $\x_{i+1} > \x_i$.
\end{proof}

\begin{lemma}\label{lemma:7}
\begin{align*}
(\lknot_k,\rknot_{k}) \cap (\x_{i-1},\x_i) &=
\begin{cases}
  (\x_{i-1},\x_i), & k = \rintsum(i)-\p_i, \ldots,  \rintsum(i), \\
  \emptyset, & \text{otherwise}.
\end{cases}
\end{align*}
\end{lemma}
\begin{proof}
It follows from the previous lemma that the intervals $(\lknot_k, \rknot_{k})$, $k=1, \ldots, \N$, are non-empty and satisfy
\begin{align*}
	\bigcap_{k=k_1}^{k_2} (\lknot_k, \rknot_{k}) = (\lknot_{k_2}, \rknot_{k_1}).
\end{align*}
The minimum $k_1$ and maximum $k_2$ at which $(\lknot_{k_2}, \rknot_{k_1}) = (\x_{i-1}, \x_i)$ can be found by inspecting the knot vectors \cref{eq:knots1,eq:knots2}. It follows that $k_2 = \rintsum(i)$ and $k_1 = \lintsum(i-1)+1 = \rintsum(i) - \p_i$.
\end{proof}

In order to measure the local continuity of B-spline-like basis functions at break points, we define the following quantities. For $k=1,\ldots,\N$, let $\lknot_k = \x_i$ and $\rknot_k = \x_j$, and set
\begin{equation}
\label{eq:local-supersmoothness-0}
\begin{aligned}
\lsmooth(k) &:= \p_{i+1} -1 - \max\{l\geq 0: \lknot_{k} = \lknot_{k+l}\}, \\
\rsmooth(k) &:= \p_{j} -1 - \max\{l\geq0: \rknot_{k} = \rknot_{k-l}\}.
\end{aligned}
\end{equation}
Note that from the knot vector definitions \cref{eq:knots1,eq:knots2} it can be deduced that
\begin{equation}\label{eq:local-supersmoothness}
\begin{aligned}
\lsmooth(k) &= \r_i + \max\{l\geq 0: \lknot_{k} = \lknot_{k-l}\} \geq r_i, \\
\rsmooth(k) &= \r_j + \max\{l\geq0: \rknot_{k} = \rknot_{k+l}\} \geq r_j.
\end{aligned}
\end{equation}
With this notation in place, we can give a local dimension formula.

\begin{lemma}\label{lem:local-dim}
The restriction of the spline space $\splSpacerp$ to the interval $[\lknot_k, \rknot_{k}] $ has dimension $\lsmooth(k)+\rsmooth(k)+3$.
\end{lemma}
\begin{proof}
Let $\lknot_k = \x_i$ and $\rknot_k = \x_j$, $j>i$. The restriction of the spline space $\splSpacerp$ to the interval $[\lknot_k, \rknot_{k}]$ has dimension
\begin{equation*}
\sum_{l=i+1}^{j}(\p_{l}+1)-\sum_{l=i+1}^{j-1}(\r_l+1)=\p_j+1+\sum_{l=i+1}^{j-1}(\p_{l}-\r_l).
\end{equation*}
On the other hand, from \cref{eq:knots1,eq:knots2,eq:local-supersmoothness-0}
we get
\begin{align*}
k+\p_{i+1}-1-\lsmooth(k)&=k+\max\{l\geq 0: \lknot_{k} = \lknot_{k+l}\}=\sum_{l=0}^{i}(\p_{l+1}-\r_l),
\\
k-\p_j+1+\rsmooth(k)&=k-\max\{l\geq0: \rknot_{k} = \rknot_{k-l}\}=1+\sum_{l=1}^{j-1}(\p_l-\r_l),
\end{align*}
thus
\begin{equation*}
k-\lsmooth(k)=2+\sum_{l=1}^{i}(\p_{l}-\r_l), \quad k+\rsmooth(k)=\sum_{l=1}^{i}(\p_{l}-\r_l)+\sum_{l=i+1}^{j-1}(\p_{l}-\r_l)+\p_j.
\end{equation*}
Subtracting the above expressions gives the result.
\end{proof}

We now show the existence of a basis of the space $\splSpacerp$ with several nice properties.

\begin{theorem}[Unit-integral B-splines] \label{thm:M-splines}
Assume there exist admissible weights for the space $\splSpacerp$. Then, there exists a basis of the space $\splSpacerp$ consisting of the functions $\{\MS_k, \; k=1,\ldots,\N\}$, with the following properties:
\begin{align}
& \MS_k(\x) > 0, 	\quad \text{for all }  \x \in  (\lknot_k, \rknot_{k}), 					& (\text{Non-negativity})		\label{eq:support1M} 	\\
& \MS_k(\x) = 0, 	\quad \text{for all }  \x \notin [\lknot_k, \rknot_{k}],			& (\text{Local support})			\label{eq:support2M}	\\
& \int_{\lknot_k}^{\rknot_{k}} \MS_k(\x)\dint\x = 1, 									& (\text{Unit integral})			\label{eq:uiM}	\\
& \MS_k \text{ is exactly } C^{\lsmooth(k)}  \text{ at } \x = \lknot_k,		& (\text{Start-point smoothness})	\label{eq:continuity1M}	\\
& \MS_k \text{ is exactly } C^{\rsmooth(k)}  \text{ at } \x = \rknot_k, 		& (\text{End-point smoothness})	\label{eq:continuity2M} \\
& \Span{\left. \MS_{k} \right|_{\I_i}}_{k=\rintsum(i)-\p_i}^{\rintsum(i)} \equiv \ECT{\p_i}{(i)}.		& (\text{Local linear independence})  \label{eq:basisM}
\end{align}
\end{theorem}
\begin{proof}
Consider the restriction of the spline space $\splSpacerp$ to the interval $[\lknot_k, \rknot_{k}] $. Since there exist admissible weights for $\splSpacerp$, we know that \cite[Theorem~3.1]{Buchwald:2003} ensures the existence of a unique function $s_k$ in such space satisfying the following Hermite interpolation problem:
\begin{align*}
D^{j}_{+} s_k(\lknot_k) &= 0, \quad j=0, \ldots, \lsmooth(k), \\
D^{\lsmooth(k)+1}_{+} s_k(\lknot_k) &=1,  \\
D^{j}_{-} s_k(\rknot_k) &= 0, \quad j=0, \ldots, \rsmooth(k).
\end{align*}
Moreover, from \cref{lem:local-dim} and \cite[Theorem~2.5]{Buchwald:2003} it follows that the function $s_k$ has no additional zeros (counting multiplicity) in $[\lknot_k, \rknot_{k}]$; see also \cite[Theorem~4.2]{Nurnberger:1984}. Therefore, taking
\begin{equation*}
\MS_k(\x):=
\begin{dcases}
\frac{s_k(\x)}{\int_{\lknot_k}^{\rknot_{k}} s_k(\y)\dint\y },& \lknot_k\leq \x<\rknot_{k},
\\
0, & \text{otherwise},
\end{dcases}
\end{equation*}
it can be easily checked that $\MS_k$ satisfies \cref{eq:support1M,eq:support2M,eq:uiM,eq:continuity1M,eq:continuity2M}.
Note that \cref{eq:local-supersmoothness} implies that $\MS_k$ can have a higher smoothness at the end points of its support ($\lknot_k$ and $\rknot_{k}$) than required by the space $\splSpacerp$.

With the same line of arguments of \cite[Theorem 4.4]{Nurnberger:1984} we can show that the  functions in the set ${\cal M}:=\{\MS_k, \; k=1,\ldots,\N\}$ are linearly independent and thus span the space $\splSpacerp$.
Finally, \cref{lemma:7} shows that $\{M_k, \; k=\rintsum(i)-\p_i, \ldots, \rintsum(i)\}$ are the only functions in ${\cal M}$ which are non-zero functions on the interval $[\x_{i-1}, \x_i)$. Consequently, these $\p_i+1$ functions span \ECT{\p_i}{(i)} and thus  they are locally linearly independent.
\end{proof}

Under a proper assumption on the weights related to the local ECT-spaces, one can also define a basis that forms a partition of unity, so
mimicking all properties of the standard polynomial B-spline basis obtained by the Cox--de Boor recursion formula \cite{Deboor:1978}.

\begin{theorem}[GTB-splines] \label{thm:B-splines}
Assume there exist admissible weights for the space $\splSpacerp$ and for all the \ECT{\p_i}{(i)} we have
\begin{equation}
\label{eq:weights-pou}
w_0^{(i)} =1, \quad i=1,\ldots,\nelms.
\end{equation}
Then, there exists a basis of the space $\splSpacerp$ consisting of the functions $\{\BS_k, \; k=1, \ldots, \N\}$, with the following properties:
\begin{align}
& \BS_k(\x) > 0, 	\quad \text{for all }  \x \in  (\lknot_k, \rknot_{k}), 					& (\text{Non-negativity})		\label{eq:support1} 	\\
& \BS_k(\x) = 0, 	\quad \text{for all }  \x \notin [\lknot_k, \rknot_{k}], 					& (\text{Local support})			\label{eq:support2}	\\
& \sum_{k=1}^\N \BS_k(\x) = 1, \quad \text{for all }	\x \in[\a,\b],								& (\text{Partition of unity})			\label{eq:pou}	\\
& \BS_k \text{ is exactly } C^{\lsmooth(k)}  \text{ at } \x = \lknot_k,		& (\text{Start-point smoothness})	\label{eq:continuity1}	\\
& \BS_k \text{ is exactly } C^{\rsmooth(k)}  \text{ at } \x = \rknot_k, 		& (\text{End-point smoothness})	\label{eq:continuity2} \\
& \Span{\left. \BS_{k} \right|_{\I_i}}_{k=\rintsum(i)-\p_i}^{\rintsum(i)} \equiv \ECT{\p_i}{(i)}.		& (\text{Local linear independence})  \label{eq:basis}
\end{align}
\end{theorem}
\begin{proof}
Without loss of generality, we can assume $\r_i\geq 0$, $i=1,\ldots, \nelms-1$.
Indeed, if $\r_l=-1$ for some $1\leq l\leq \nelms-1$ then the spline space $\splSpacerp$ can be decomposed into two disconnected spaces defined on $[\x_0,\x_l)$ and $[\x_l,\x_\nelms]$, respectively.
Recall that $\r_i\geq 0$, $i=1,\ldots, \nelms-1$, implies that $\lknot_k<\rknot_{k-1}$, $k=1,\ldots,\N$ (see \cref{lemma:6}).
Furthermore, if $\p_i=0$ then $\BS_k(\x)=1$ for $\x\in[\x_{i-1},\x_i)\subseteq[\lknot_k, \rknot_{k})$.
For $i=1,\ldots,\nelms$, let
\ECTtilde{\p_i}{(i)} denote the space of the derivatives of the functions belonging to $\ECT{\p_i}{(i)}$. For $\p_i>0$, from \cref{eq:weights-pou,eq:gen-powers-der} it follows that
\ECTtilde{\p_i}{(i)} is an ECT-space of dimension $\p_i$ on $\I_i$ identified by the weights
\begin{equation}\label{adm-w:der}
\widehat w_j^{(i)}:=w_{j+1}^{(i)}\in C^{\p_i-1-j}, \quad j=0 \ldots,\p_i-1.
\end{equation}

Consider now the spline space
\begin{equation} \label{eq:spline-space-tilde}
\begin{aligned}
\splSpacerpDer :=  \Big\{ {\Rd s:} \,& {\Rd [\a, \b] \rightarrow \RR :}
\left. s \right|_{\I_i} \in \ECTtilde{\p_i}{(i)},\; i = 1,\ldots,\nelms  \;\text{ and } \; \\
&\jump{\x_i, j}{s} = 0,\; j = 0,\ldots,\r_i-1, \; i = 1,\ldots,\nelms-1   \Big\},
\end{aligned}
\end{equation}
which has dimension $\N-1$. 
Since the weights defined in \cref{adm-w:der} are admissible for $\splSpacerpDer$,
we can apply \cref{thm:M-splines} to all (non-trivial) disconnected parts of $\splSpacerpDer$, and so the space admits a basis of functions $\{\MStilde_k, \; k=2,\ldots,\N\}$ such that
\begin{align}
& \MStilde_k(\x) > 0, 	\quad \text{for all }  \x \in  (\lknot_k, \rknot_{k-1}), 	\label{eq:posMtilde}	    \\
& \MStilde_k(\x) = 0, 	\quad \text{for all }  \x \notin [\lknot_k, \rknot_{k-1}],	 \label{eq:support1Mtilde}		\\
& \int_{\lknot_k}^{\rknot_{k-1}} \MStilde_k(\x)\dint\x = 1, 				             \label{eq:uiMtilde}    \\
& \MStilde_k \text{ is exactly } C^{\lsmooth(k)-1}  \text{ at } \x = \lknot_k,	 \label{eq:continuity1Mtilde}		\\
& \MStilde_{k} \text{ is exactly } C^{\rsmooth(k-1)-1}  \text{ at } \x = \rknot_{k-1}. 	  \label{eq:continuity2Mtilde}
\end{align}
We then define the following functions belonging to $\splSpacerp$:
\begin{align*}
\BS_1(\x) &:= 1-\int_{\a}^\x \MStilde_{2}(\y)\dint\y,
\\
\BS_k(\x) &:= \int_{\a}^\x \MStilde_k(\y)\dint\y - \int_{\a}^\x \MStilde_{k+1}(\y)\dint\y, \quad  k=2, \ldots, \N-1,
\\
\BS_\N(\x) &:= \int_{\a}^\x \MStilde_\N(\y)\dint\y.
\end{align*}
A direct inspection shows that the above functions satisfy \cref{eq:pou}. From  \cref{eq:support1Mtilde,eq:continuity1Mtilde,eq:continuity2Mtilde}, we deduce \cref{eq:continuity1,eq:continuity2}. Moreover, \cref{eq:continuity1Mtilde,eq:support1Mtilde} imply
\begin{equation}
\label{eq:pos-der}
D^{\lsmooth(k)+1}_{+}\BS_k(\lknot_k) >0.
\end{equation}
 For $k=1,\ldots, \N$, from \cite[Theorem 2.5]{Buchwald:2003} it follows that the function $\BS_k$ has no additional zeros (counting multiplicity) in $[\lknot_k, \rknot_{k}]$. Then, \cref{eq:pos-der} gives \cref{eq:support1}.
 Linear independence of the functions $\{\BS_k, \; k=1, \ldots, \N\}$  follows  by applying the same  line of arguments as in the proof of \cref{thm:M-splines}.
\end{proof}

\begin{remark}
The assumption in \cref{eq:weights-pou} is equivalent to the fact that each $\ECT{\p_i}{(i)}$ contains the constants for $i=1,\ldots,m$.
\end{remark}

\begin{remark}\label{eq:support-characterisation}
Each of the splines $\MS_k$ and $\BS_k$, $k=1, \ldots \N$, in the space $\splSpacerp$ are uniquely defined, up to a constant multiple, by a triple $\tripleA{k}{k}$.
Indeed, from \cref{lem:local-dim} it follows that the end-point smoothness conditions (see \cref{eq:continuity1M}--\cref{eq:continuity2M} and \cref{eq:continuity1}--\cref{eq:continuity2}) uniquely determine the spline function, up to a constant multiple. The additional constraint of unit integral \cref{eq:uiM} or partition of unity \cref{eq:pou}, respectively, specifies this constant.
\end{remark}

Properties \cref{eq:support1}--\cref{eq:basis} are very important in both geometric modeling and isogeometric analysis; they make the set of GTB-splines $\{\BS_k, \; k=1, \ldots, \N\}$ the basis of choice for the space $\splSpacerp$ in those applications. Although the theory of GTB-splines has been established for many years \cite{Buchwald:2003, Nurnberger:1984}, a stable and efficient way for computing the GTB-spline basis functions, and performing fundamental operations such as knot insertion, has been lacking.
The next sections, containing the original contribution of the paper, focus on these issues. We start by describing some integral recurrence relations that are suited for symbolic computation, and afterwards we develop a procedure based on knot insertion that is suited for numerical evaluation. 

%% file: sections/section3_rec.tex
\section{Integral recurrence relations} \label{sec:3_rec}
In this section we present some integral recurrence relations which could be used for symbolic computation of the GTB-spline basis $\{\BS_k, \; k=1, \ldots, \N\}$. From the proof of \cref{thm:B-splines} we know that they can be obtained by integrating certain unit-integral B-splines of lower degree and smoothness $\{\MStilde_k, \; k=2,\ldots,\N\}$. We also note that any set $\{\BStilde_k, \; k=2,\ldots,\N\}$ of functions  
in the same space satisfying \cref{eq:posMtilde,eq:support1Mtilde,eq:continuity1Mtilde,eq:continuity2Mtilde} can be easily converted into the unit-integral B-splines by
\begin{equation*}
\MStilde_k(\x)=\frac{\BStilde_k(\x)}{\int_{\a}^{\b} \BStilde_{k}(\y)\dd\y}.
\end{equation*}
These form the ingredients for the considered integral recurrence relations. With $\p := \max_{1\leq i \leq\nelms} \p_i$, we will discuss in the following how to recursively construct the GTB-splines $\BS_k = \BS_{k,\p}$, $k=1, \ldots, \N$, under the same assumptions as in \cref{thm:B-splines}. We will consider a local and a global recursive construction.

We start with a pointwise recurrence. It is a direct extension of the integral relation presented in \cite{Beccari:2017,Toshniwal:2018b} for polynomial splines of non-uniform degree.
For $\q=0,\ldots, \p$ and $k=\p-\q+1,\ldots,\N$, the spline $\BS_{k,\q}$ is supported on the interval $[\lknot_k,\rknot_{k - \p+\q}]$ and defined at $\x \in [\x_{i-1}, \x_{i}) \subset [\lknot_k,\rknot_{k - \p+\q}]$ as follows:
\begin{equation}\label{eq:rec-local}
\BS_{k,\q}(\x) :=
\begin{dcases}
    w_{\p_i}^{(i)}(\x), & \q = \p - \degree_{i},\\
    w_{\p-\q}^{(i)}(\x)\int_{\a}^\x \biggl[\dfrac{\BS_{k,\q-1}(\y)}{\bspint_{k,\q-1}} - \dfrac{\BS_{k+1,\q-1}(\y)}{\bspint_{k+1,\q-1}}\biggr] \dd \y, & \q > \p - \degree_{i},\\
    0, & \text{otherwise},
\end{dcases}
\end{equation}
where
\begin{equation}\label{eq:integral}
\bspint_{j,\q-1} := \int_{\a}^{\b} \BS_{j, \q-1}(\y)\dd\y.
\end{equation}
In the above we assumed that any undefined $\BS_{j, \q-1}$ with $j<\p-\q+2$ or $j>\N$ must be regarded as the zero function,
and we used the convention that if $\bspint_{j, \q-1}=0$ then
\begin{equation} \label{eq:limit}
\int_{\a}^\x \dfrac{\BS_{j,\q-1}(\y)}{\bspint_{j,\q-1}}\dd\y :=
\begin{cases}
1, & \x \geq \lknot_{j} \text{~and~} j\leq \N,\\
0, & \text{otherwise}.
\end{cases}
\end{equation}

The relation in \cref{eq:rec-local} builds up the GTB-splines on each of the intervals $[\x_{i-1}, \x_{i})$ separately; hence, it is a \emph{local recurrence}.
When all the degrees $\p_i$ are uniform, GTB-splines are usually defined by means of a global recurrence formula; see, e.g., \cite{Lyche:2019}.
In order to produce a \emph{global recurrence}, we first define a global set of weight functions $\{w_j, \; j=0,\ldots,\p\}$ by
\begin{equation} \label{eq:weights-global}
w_j(\x) :=
\begin{cases}
    w_{j}^{(i)}(\x), & j\leq \p_i, \\
    0, & \text{otherwise},
\end{cases}
\quad \x \in [\x_{i-1}, \x_{i}), \quad i=1,\ldots,\nelms.
\end{equation}
This allows us to redefine $\BS_{k,\q}$ globally. For $\q=0,\ldots, \p$ and $k=\p-\q+1,\ldots,\N$, the spline $\BS_{k,\q}$ can be evaluated at $\x \in [\a,\b)$ as follows:
\begin{equation} \label{eq:rec-global-0}
\BS_{k, 0}(\x) :=
\begin{cases}
    w_{\p}(\x), & \x \in [\x_{i-1}, \x_{i}),\\
    0, & \text{otherwise},
\end{cases}
\end{equation}
and
\begin{equation}\label{eq:rec-global-q}
\BS_{k,\q}(\x) :=
    w_{\p-\q}(\x)\int_{\a}^\x \biggl[\dfrac{\BS_{k,\q-1}(\y)}{\bspint_{k,\q-1}} - \dfrac{\BS_{k+1,\q-1}(\y)}{\bspint_{k+1,\q-1}}\biggr] \dd \y, \quad \q > 0,
\end{equation}
where $\bspint_{j, \q-1}$ is defined in \cref{eq:integral}. Again, we assumed that any undefined $\BS_{j, \q-1}$ with $j<\p-\q+2$ or $j>\N$ must be regarded as the zero function, and we used the convention that if $\bspint_{j, \q-1}=0$ then \cref{eq:limit} is taken.
At the right end point $\b$, the spline  $\BS_{k,\q}$ is defined by taking the limit from the left, i.e., $\BS_{k,\q}(b):=\lim_{\x\rightarrow b,\x<b}\BS_{k,\q}(\x)$.

\begin{remark}
The global recurrence \cref{eq:rec-global-0}--\cref{eq:rec-global-q} is exactly the same as the definition known for \Chef B-splines of uniform local dimensions; see \cite[Definition~7]{Lyche:2019}. This was possible thanks to the enrichment of the local weights with zero functions so that we have $\p+1$ functions associated with each local interval $[\x_{i-1}, \x_{i})$, and we can glue them together into the global weights in \cref{eq:weights-global}.
\end{remark}

\begin{remark}
The above recurrence relations are stated for the GTB-splines $\BS_k$ but they do not require  partition of unity. Actually, they can be used to construct locally supported  functions enjoying properties \cref{eq:support1}, \cref{eq:support2}, and \cref{eq:continuity1}--\cref{eq:basis} that sum up to $w_0$; see \cite{Lyche:2019} for \Chef B-splines of uniform local dimensions.
\end{remark}

We note that for $\nelms=1$ the space $\splSpacerp$ reduces to a single ECT-space and the GTB-spline basis always exists, provided that \cref{eq:weights-pou} holds. In analogy with the polynomial case, this basis is called the \highlight{Bernstein basis} corresponding to the considered ECT-space.
We end this section by considering the special spline space $\splSpacerp=\splSpacep$ of discontinuous GT-splines. In this case, under assumptions \cref{eq:weights-pou}, the GTB-spline basis always exists because all weight systems identifying the various ECT-spaces $\ECT{\p_i}{(i)}$ are admissible for $\splSpacep$. This basis is nothing else than the global Bernstein basis.

\begin{definition}[Global Bernstein basis] \label{eq:Bernstein-global}
Let $\{\bs^{(i)}_j, \; j=0,\ldots,\p_i\}$ be the Bernstein basis corresponding to the ECT-space $\ECT{\p_i}{(i)}$, $i=1,\ldots,\nelms$. Let $l := l(i,j) := \sum_{k=1}^{i-1} (\p_k+1) + j$ and define
\begin{equation*}
\bs_{l}(\x) :=
\begin{cases}
\bs^{(i)}_{j}(\x), & \x \in \I_i, \\
0,                & \text{otherwise.}
\end{cases}
\end{equation*}
\end{definition}

We recall from \cite[Example 16]{Lyche:2019} that the local Bernstein functions $\bs^{(i)}_{j} := \bs_{j,\p_{i}}$, $j=0, \ldots, \p_i$ are defined recursively as follows. For $\q=0,\ldots, \p_i$ and $j=0,\ldots,\q$, the function $\bs_{j,\q}$ is defined at $\x \in [\x_{i-1}, \x_{i}]$ as
\begin{equation} \label{eq:rec-bern-0}
\bs_{0,0}(\x) := w_{\p_i}^{(i)}(\x),
\end{equation}
and
\begin{equation}  \label{eq:rec-bern-q}
\bs_{j,\q}(\x) := w_{\p_i-\q}^{(i)}(\x)\begin{dcases}
1-\int_{\x_{i-1}}^\x \dfrac{\bs_{0,\q-1}(\y)}{\bbint_{0,\q-1}}\dint\y, & j=0,
\\
\int_{\x_{i-1}}^\x \biggl[\dfrac{\bs_{j-1,\q-1}(\y)}{\bbint_{j-1,\q-1}} - \dfrac{\bs_{j,\q-1}(\y)}{\bbint_{j,\q-1}}\biggr]\dint\y, &  0< j < q,
\\
\int_{\x_{i-1}}^\x \dfrac{\bs_{\q-1,\q-1}(\y)}{\bbint_{\q-1,\q-1}}\dint\y, & j=\q,
\end{dcases}\quad \q>0,
\end{equation}
where
\begin{equation}\label{eq:bern-integral}
\bbint_{j,\q-1} := \int_{\x_{i-1}}^{\x_i} \bs_{j, \q-1}(\y)\dd\y.
\end{equation}

\begin{remark}
\label{rmk:Hermite-Bernstein}
Instead of using the recurrence relation \cref{eq:rec-bern-0}--\cref{eq:rec-bern-q}, each Bernstein function $\bs^{(i)}_{j}$ can also be computed by solving the following  Hermite interpolation problem in the space $\ECT{\p_i}{(i)}$: for $j=0$,
\begin{equation*}
  \bs^{(i)}_{0}(\x_{i-1})=1, \quad
  D^{l}\bs^{(i)}_{0}(\x_{i})=0, \quad l=0,\ldots,\p_i-1,
\end{equation*}
and for $j=1,\ldots, \p_i$,
\begin{align*}
  D^{l}\bs^{(i)}_{j}(\x_{i-1})&=0, \quad  l=0,\ldots,j-1, \quad
  D^{l}\bs^{(i)}_{j}(\x_{i})=0, \quad l=0,\ldots,\p_i-j-1, \\
  D^{j}\bs^{(i)}_{j}(\x_{i-1})&=-\sum_{l=0}^{j-1}D^{j}\bs^{(i)}_{l}(\x_{i-1}).
\end{align*}
Since $\ECT{\p_i}{(i)}$ is an ECT-space, this interpolation problem has a unique solution; see \cref{def:ET}. Any convenient basis in $\ECT{\p_i}{(i)}$ can be used to represent the  Bernstein functions.
\end{remark}

 The presented integral recurrence relations, in particular the global one in \cref{eq:rec-global-0}--\cref{eq:rec-global-q}, are suited for symbolic computation. However, they might lack stability in numerical computation. In the next section we provide a knot insertion procedure which is an important ingredient to produce an efficient and stable numerical evaluation algorithm for the basis functions we are interested in.

%% file: sections/section4.tex
\section{Knot insertion} \label{sec:4}
Knot insertion is the fundamental operation of inserting a new knot into an existing knot vector while maintaining the shape of a spline curve. If the new knot is already present in the initial knot vector, then the continuity is reduced at the corresponding breakpoint. Otherwise, a new breakpoint is inserted in the partition.

\begin{lemma} \label{lem:jump}
If $\r_i<\min\{\p_i,\p_{i+1}\}$, then
there are exactly $\r_i+3$ successive GTB-splines that have a jump in their $(\r_i+1)$-th order derivative  at $\x = \x_i$:
\begin{align*}
	& \jump{\x_i, \r_i+1}{\BS_k} = 0, \quad \text{for } k = 1, \ldots, \lintsum(i)-1, \\
	& \jump{\x_i, \r_i+1}{\BS_k}  \neq 0, \quad \text{for } k = \lintsum(i), \ldots , \rintsum(i)+1, \\
	& \jump{\x_i, \r_i+1}{\BS_k}  = 0, \quad \text{for } k = \rintsum(i)+2, \ldots , \N.
\end{align*}
\end{lemma}
\begin{proof} The local support \cref{eq:support2} implies that all GTB-splines vanish in a neighborhood of $\x_i$ except for $\BS_{k}$, $k=\lintsum(i-1)+1, \ldots,  \rintsum(i+1)$. From \cref{eq:continuity2} it follows that the GTB-splines $\BS_{\lintsum(i)-\lmult}(\x)$, with $\lmult>0$,  are at least $C^{\r_i+1}$-smooth at $\x = \x_i$. Similarly, from \cref{eq:continuity1} it follows that the GTB-splines $\BS_{\rintsum(i)+1+\rmult}(\x)$, with $\rmult>0$, are at least $C^{\r_i+1}$-smooth at $\x = \x_i$. The remaining GTB-splines $\BS_k(\x)$, $k = \lintsum(i), \ldots , \rintsum(i)+1$ are $C^{\r_i}$-smooth at $\x = \x_i$.
From \cite[Theorem 4.2]{Buchwald:2003} and \cite[Theorem 4.3]{Nurnberger:1984} it follows that these functions have minimal support, thus they cannot be smoother at $\x = \x_i$.
Since $\rintsum(i)+1 - (\lintsum(i)-1) = \r_i+2 - \r_0 = \r_i+3$ the result follows.
\end{proof}

Suppose we remove a knot $\lknot = \rknot = \x_i \in (a, b)$ from $\mbf{\lknot}{}$ and $\mbf{\rknot}{}$, respectively, resulting in a new spline space $\splSpacerpB$ with corresponding knot vectors,
\begin{align*}
	\tilde{\mbf{\lknot}{}} := (\tilde{\lknot}_k )_{k=1}^{\N-1}:= ( \;
		\underbrace{\x_0,\;  \ldots,\;  \x_0}_{\p_{1} - \r_0 \text{ times}},\; \ldots,\;
		&\underbrace{\x_{i-1},\; \ldots,\;  \x_{i-1}}_{\p_{i} - \r_{i-1} \text{ times}},\;
		\underbrace{\x_{i},\; \ldots,\; \x_{i}}_{\p_{i+1} - \r_i -1 \text{ times}},\;
		\\
		&\underbrace{\x_{i+1},\;  \ldots,\; \x_{i+1}}_{\p_{i+2} - \r_{i+1} \text{ times}},\; \ldots,\;
		\underbrace{\x_{\nelms-1},\; \ldots,\; \x_{\nelms-1}}_{\p_{\nelms} - \r_{\nelms-1} \text{ times}}
	\; ),
\end{align*}
and
\begin{align*}
	\tilde{\mbf{\rknot}{}} := (\tilde{\rknot}_k )_{k=1}^{\N-1}:= (\;
		\underbrace{\x_1,\; \ldots,\; \x_1}_{\p_1 - \r_1 \text{ times}},\; \ldots,\;
		&\underbrace{\x_{i-1},\; \ldots,\; \x_{i-1}}_{\p_{i-1} - \r_{i-1} \text{ times}},\;
		\underbrace{\x_{i},\; \ldots,\; \x_{i}}_{\p_i - \r_i-1 \text{ times}},\;\\
		&\underbrace{\x_{i+1},\; \ldots,\; \x_{i+1}}_{\p_{i+1} - \r_{i+1} \text{ times}},\; \ldots,\;
		\underbrace{\x_\nelms,\; \ldots,\; \x_\nelms}_{\p_{\nelms} - \r_{\nelms} \text{ times}}
	\; ).
\end{align*}
The smoothness vector is easily deduced to be $\tilde{\bsmooth} = (\smooth_0, \ldots, \smooth_{i-1}, \smooth_{i}+1, \smooth_{i+1}, \ldots, \smooth_{\nelms})$; it is assumed to satisfy the same restrictions as in \cref{eq:smoothness}, and so $\smooth_i+1\leq\min\{\p_i,\p_{i+1}\}$. Consequently, the spline space $\splSpacerpB$ is a subspace of $\splSpacerpA$ with one additional continuous derivative at $\x = \x_i$.

From \cref{eq:support-characterisation} we recall that each of the GTB-splines $\tilde{\BS}_k$, $k=1, \ldots \N-1$, that form a basis for $\splSpacerpB$, are uniquely defined, up to a constant multiple, by a triple $\tripleB{k}{k}$.  It can directly be verified that the following relation is consistent with the definition of the knot vectors $\tilde{\mbf{\lknot}{}}$ and $\tilde{\mbf{\rknot}{}}$.

\begin{lemma} \label{lem:tripple2}
The following relationship holds
\begin{equation*}
	\tripleB{k}{k} =
	\begin{cases}
		\tripleA{k}{k}, & \text{if } 1 \leq k < \lintsum(i),\\
		\tripleA{k}{k+1}, & \text{if } \lintsum(i) \leq k \leq \rintsum(i),	\\
		\tripleA{k+1}{k+1}, \hspace*{-0.05cm}
		& \text{if } \rintsum(i) < k < \N.
	\end{cases}
\end{equation*}
\end{lemma}

An alternative, yet, equivalent perspective is that $\mbf{\lknot}$ and $\mbf{\rknot}$ are obtained from $\tilde{\mbf{\lknot}{}}$ and $\tilde{\mbf{\rknot}{}}$ by the process of knot insertion.

\begin{proposition}\label{prop:knot_insertion} Let $\mbf{\lknot}{} $ and $\mbf{\rknot}{}$ be obtained from $\tilde{\mbf{\lknot}{}} $ and $\tilde{\mbf{\rknot}{}} $ by inserting a single knot $\lknot = \rknot = \x_i \in (a, b)$, respectively. Then,
\begin{equation}\label{eq:knotremoval}
\tilde{\BS}_k(\x) = \ca_k \BS_k(\x) + \cb_{k+1} \BS_{k+1}(\x),
\end{equation}
where
\begin{itemize}
\item[(i)] $\ca_{k} = 1$ and $\cb_{k+1}=0$ if $1\leq k < \lintsum(i)$;
\item[(ii)] $\ca_{k} > 0$ and $\cb_{k+1} = - \ca_{k} \dfrac{\jump{\x_i, \r_i+1}{\BS_{k}}}{\jump{\x_i, \r_i+1}{\BS_{k+1}}} > 0$ if $\lintsum(i) \leq k \leq \rintsum(i)$;
\item[(iii)] $\ca_{k} = 0$ and $\cb_{k+1}=1$ if $\rintsum(i) < k < \N$.
\end{itemize}
\end{proposition}
\begin{proof}
Because $\splSpacerpB \subset \splSpacerpA$, every $\tilde{\BS}_k \in \splSpacerpB$ can be uniquely written as a linear combination of the GTB-splines that form a basis for $\splSpacerpA$. The particular functions involved in the linear combination in \cref{eq:knotremoval} follow from \cref{lem:tripple2}. Case (i) and (iii) follow directly. Case (ii) follows from \cref{lem:jump} and
\begin{equation} \label{eq:constraint-jump}
	0= \jump{\x_i, \r_i+1}{\tilde{\BS}_k} = \ca_k \jump{\x_i, \r_i+1}{\BS_k} + \cb_{k+1} \jump{\x_i, \r_i+1}{ \BS_{k+1}}.
\end{equation}

Finally, the start-point smoothness \cref{eq:continuity1} implies that
\begin{equation*}
	D^{j}_{+} \tilde{\BS}_k(\lknot_k) = D^{j}_{+} \BS_k(\lknot_k)=0, \quad j=0,\ldots,\lsmooth(k),
\end{equation*}
and
\begin{equation}\label{eq:max-jump}
	D^{\lsmooth(k)+1}_{+} \tilde{\BS}_k(\lknot_k) = \ca_{k} D^{\lsmooth(k)+1}_{+} \BS_k(\lknot_k) \neq 0.
\end{equation}
Then, the positivity \cref{eq:support1} implies that the derivatives on both sides of \cref{eq:max-jump} have the same sign. Consequently, $\ca_{k}$ must be positive. A similar argument, involving the end-point smoothness in \cref{eq:continuity2}, shows that $\cb_{k+1}$ is positive.
\end{proof}

We can also write \cref{eq:knotremoval} in matrix notation,
\begin{equation}
\label{eq:knotremovalmatrix}
	 \tilde{\BS}_k(\x) = \sum_{l=1}^{\N} \C_{kl} \BS_l(\x) , \quad k=1,\ldots, \N-1 \quad \Longleftrightarrow \quad
	\vect{\tilde{\BS}}{}(\x) = \C \; \vect{\BS}{}(\x).
\end{equation}
Here, $\C \in \RR^{(\N-1) \times \N}$ and the entries $\C_{kl}$ are determined by the three cases (i)--(iii). The matrix $\C$ has the following sparsity structure:
\begin{equation}\label{eq:extract1}
\C =
\begin{bmatrix}
\mat{I}_A & 			&\\
        & \hat{\mat{C}} & \\
        & & \mat{I}_B
\end{bmatrix}, \quad
\hat{\CC} =
\begin{bmatrix}
\ca_{\lintsum(i)} 	& \cb_{\lintsum(i)+1} 		& 				&  					&				\\
                & \ca_{\lintsum(i)+1}  	& \ddots			& 					&				\\
                &   					&  \ddots 			&  	\cb_{\rintsum(i)}		&				\\
                &  					&  				&   	\ca_{\rintsum(i)} 	& \cb_{\rintsum(i)+1}		
\end{bmatrix}.
\end{equation}
Here, $\mat{I}_A$ and $\mat{I}_B$ are identity matrices of size $(\lintsum(i)-1)$ and $(\N-1-\rintsum(i))$, respectively, and $\hat{\CC} \in \RR^{(\r_i+2) \times (\r_i+3)}$.

\begin{theorem}[Knot insertion] \label{thm:knotinsertion} Let $\mbf{\lknot}{} $ and $\mbf{\rknot}{}$ be obtained from $\tilde{\mbf{\lknot}{}} $ and $\tilde{\mbf{\rknot}{}} $ by inserting a single knot $\lknot = \rknot = \x_i \in (a, b)$, respectively. Then,
\begin{equation}\label{eq:knotremoval1}
	s(\x) = \sum_{k=1}^{\N-1} \tilde{\dof}_k \tilde{\BS}_k(\x) = \sum_{k=1}^{\N} \dof_k \BS_k(\x),
\end{equation}
where
\begin{equation}\label{eq:knotremoval2}
	\dof_k =
	\begin{cases}
		\newdof_k, &  1 \leq k \leq \lintsum(i),\\
		\cb_{k} \newdof_{k-1} + \ca_{k} \newdof_{k}, &  \lintsum(i) < k \leq \rintsum(i), \\
		\newdof_{k-1}, &  \rintsum(i) < k \leq \N, \\
	\end{cases}
\end{equation}
with $\ca_k + \cb_{k} = 1$.
\end{theorem}
\begin{proof}
We write \cref{eq:knotremoval1} as,
\begin{equation*}
	\vect{\newdof}{}^T \vect{\tilde{\BS}}{}(\x)  =
	\vect{\newdof}{}^T \bigl(\C \; \vect{\BS}{}(\x) \bigr) =
	\bigl(\vect{\newdof}{}^T \C \bigr) \vect{\BS}{}(\x)  =
	 \vect{\dof}{}^T \vect{\BS}{}(\x).
\end{equation*}
It follows that
\begin{equation*}
	\vect{\dof}{}^T =   \vect{\newdof}{}^T \C \quad
	 \Longleftrightarrow \quad
	 \dof_l = \sum_{k=1}^{\N-1}   \newdof_k \C_{kl}, \quad l=1,\ldots, \N.
\end{equation*}
The structure of $\C$, given in \cref{eq:extract1}, leads to the result in \cref{eq:knotremoval2}. Finally, the partition of unity \cref{eq:pou} of both bases $\{\BS_k, \; k=1,\ldots, \N\}$ and $\{\tilde{\BS}_k, \; k=1,\ldots, \N-1\}$ implies that the column sum of matrix $\C$ is one. Hence, $\ca_{\lintsum(i)}=\cb_{\rintsum(i)+1}=1$ and $\ca_k + \cb_{k} =1$ for $k = \lintsum(i)+1, \ldots , \rintsum(i)$.
\end{proof}

Besides its intrinsic interest, the knot insertion procedure can be applied  recursively in order to compute a B\'ezier extraction operator which allows for efficient and stable evaluation of GTB-splines. This will be shown in the next section.

%% file: sections/section5.tex
\section{Algorithmic evaluation} \label{sec:5}
In this section we present an algorithm that computes the GTB-spline basis $\{\BS_k, \; k=1,\ldots, \N\}$, whenever it exists, for the spline space $\splSpacerp$ using \Bezier extraction, i.e., representing each basis element in the form
\begin{equation} \label{eq:bezier-extr}
	\BS_k(\x) = \sum_{l=1}^{\M} \CC_{kl} \bs_l(\x), \quad k=1,\ldots, \N \quad
	\Longleftrightarrow \quad
    \vect{\BS}{}(\x) =  \CC \; \vect{\bs}{}(\x).
\end{equation}
Here, $\{\bs_l, \; l=1,\ldots, \M := \sum_{i=1}^\nelms (\p_i + 1)\}$ denotes the global Bernstein basis for $\splSpacep$, see \cref{eq:Bernstein-global}, and $\CC \in \RR^{\N \times \M}$ is the extraction operator that maps functions from $\splSpacep$ to $\splSpacerp$.

By construction, $\{\bs_{l}, \; l=1, \ldots, \M\}$ forms a global, locally supported basis for the space $\splSpacep$ that has the properties listed in \cref{thm:B-splines}. Since $\splSpacerp$ is a subspace of $\splSpacep$, we can use the knot insertion procedure to convert from the global Bernstein basis $\{\bs_{l}, \; l=1, \ldots, \M\}$ to the smooth GTB-spline basis $\{\BS_k, \; k=1,\ldots, \N\}$.

As already observed in \cite{Toshniwal:2018b} for the polynomial setting, the knot insertion procedure in \cref{eq:knotremoval} can be regarded as a nullspace computation of the smoothness constraints \cref{eq:constraint-jump} for all $k$ at the breakpoint $\x_i$.
Let $\vect{a}{} \in \RR^{\N}$ denote the vector with entries
\begin{equation} \label{eq:a}
\vect{a}{} :=
\begin{bmatrix}
	0 & \cdots & 0 & \jump{\x_i, \r_i+1}{\BS_{\lintsum(i)}}  & \cdots & \jump{\x_i, \r_i+1}{\BS_{\rintsum(i)+1}} & 0 & \cdots & 0
\end{bmatrix}^T.
\end{equation}
Then, the knot insertion matrix $\C \in \RR^{(\N-1) \times \N}$ in \cref{eq:knotremovalmatrix,eq:extract1} satisfies $ \C \vect{a}{}= \vect{0}{}$. This matrix can be computed by \cref{alg:1} using the vector $\vect{a}{}$ as input.

\begin{proposition}
  Given the vector $\vect{a}{}$ in \cref{eq:a} as input, \cref{alg:1} computes the matrix $\C$ in \cref{eq:extract1}.
\end{proposition}
\begin{proof}
The partition of unity \cref{eq:pou} implies that the column sum of $\C$ is equal to one. In combination with \cref{eq:constraint-jump}, it can be observed that the non-trivial entries of $\C$ can be computed in succession as follows:
\begin{align*}
	\ca_{\lintsum(i)}  &= 1 \\
	& \, \downarrow \\
	\cb_{\lintsum(i)+1} &= - \ca_{\lintsum(i)} \cdot \vect{a}{\lintsum(i)} / \vect{a}{\lintsum(i)+1} \\
	& \, \downarrow \\
	\ca_{\lintsum(i)+1} &= 1 - \cb_{\lintsum(i)+1} \\
	& \, \downarrow \\
	\cb_{\lintsum(i)+2} &= - \ca_{\lintsum(i)+1} \cdot \vect{a}{\lintsum(i)+1} / \vect{a}{\lintsum(i)+2} \\
	& \; \; \vdots \\
	\cb_{\rintsum(i)+1} &=1.
\end{align*}
It can be directly verified that $ \C \vect{a}{}=\vect{0}{}$. This logic is encoded in \cref{alg:1}.
\end{proof}

\input{algorithms/algorithm1}
\input{algorithms/algorithm2}

By applying \cref{alg:1} repeatedly the global Bernstein basis $\{\bs_{l}, \; l=1, \ldots, \M\}$ can be mapped to the smooth GTB-spline basis $\{\BS_k, \; k=1,\ldots, \N\}$. This procedure is called \highlight{\Bezier extraction} following terminology introduced in the polynomial spline context \cite{Borden:2011,Scott:2011}; we follow suit.

\begin{theorem}[\Bezier extraction] Let $\cons := \cons(i,j) := \sum_{k=1}^{i-1} (\r_k+1) + j+1$, and consider the following linear indexing of the smoothness constraints
\begin{equation*}
	\A{\cons}{\cdot} = \jump{\x_i, j}{\cdot}, \quad j = 0,\ldots,\r_i, \quad i=1,\ldots,\nelms-1.
\end{equation*}
Let the input into \cref{alg:2} be given by the matrix $\AA$ with matrix columns
\begin{equation*}
	\A{\cons}{\vect{\bs}{}},\quad \cons=1,\ldots, \O,
\end{equation*}
where the vector $\vect{\bs}{}$ collects the global Bernstein functions $\{\bs_{l}, \; l=1, \ldots, \M\}$; see \cref{eq:bezier-extr}.
Then, \cref{alg:2} produces a \Bezier extraction operator $\mat{C} \in \RR^{\N \times \M}$ that reproduces a GTB-spline basis $\{\BS_k, \; k=1,\ldots, \N\}$ for the spline space $\splSpacerpA$ according to
\begin{equation*}
	\vect{\BS}{}(\x) =  \CC \; \vect{\bs}{}(\x).
\end{equation*}
\end{theorem}
\begin{proof}
Let $\SS{(\cons)}{\mbf{p}} $ denote the GT-spline space that satisfies the first $\cons$ linear smoothness constraints $\A{k}{\cdot}$, $k = 1,\ldots, \cons$, in \cref{def:spline-space}, and let $\vect{\BS}{(\cons)}(\x)$ denote its corresponding GTB-spline basis. Note that $\SS{(0)}{\mbf{p}} \equiv \splSpacep$, $\SS{(\O)}{\mbf{p}} \equiv \splSpacerpA$ and $\SS{(\cons)}{\mbf{p}} \subset \SS{(\cons-1)}{\mbf{p}}$.

The space $\SS{(\cons-1)}{\mbf{p}} $ can be obtained from $\SS{(\cons)}{\mbf{p}} $ by inserting a single knot into its corresponding knot vectors. Consequently, there exists $\C_\cons \in \RR^{(\M-\cons) \times (\M-\cons+1)}$ with its structure given by \cref{eq:extract1} such that $\vect{\BS}{(\cons)}(\x)  = \C_{\cons} \; \vect{\BS}{(\cons-1)}(\x)$. Repeating this argument we observe that
\begin{equation*}
	\CC = \C_{\O}  \cdots \C_{2}  \C_{1}   .
\end{equation*}
\cref{alg:2} implements this recursion in line 5. It remains to show that at each step, $\cons$, the input into \cref{alg:1} is such that it reproduces operator $\C_\cons$. The correct input is given by the vector, $\vect{a}{} = \A{\cons}{\vect{\BS}{(\cons-1)}}$, such that
\begin{equation*}
	 \C_{\cons} \; \vect{a}{}  =
	 \C_{\cons} \; \A{\cons}{\vect{\BS}{(\cons-1)}} =
	 \A{\cons}{\C_{\cons}  \; \vect{\BS}{(\cons-1)}} =
	 \A{\cons}{\vect{\BS}{(\cons)}} = \vect{0}{}.
\end{equation*}
We have that
\begin{enumerate}
	\item[(i)]  at step $1$ the input into \cref{alg:1} is
	\begin{equation*}
		\A{1}{\vect{\bs}{}} = \A{1}{\vect{\BS}{(0)}};
	\end{equation*}
	\item[(ii)] the update in line 6 shows that at step $\cons$ the input into \cref{alg:1} is
	\begin{equation*}
		\C_{\cons-1} \cdots \C_1  \A{\cons}{\vect{\bs}{}} =
		 \C_{\cons-1} \cdots \C_{2} \A{\cons}{\vect{\BS}{(1)}} =
		\C_{\cons-1}  \A{\cons}{\vect{\BS}{(\cons-2)}} =
		\A{\cons}{\vect{\BS}{(\cons-1)}}.
	\end{equation*}
\end{enumerate}
Hence, by induction, the input into \cref{alg:1} is correct at every step of the recursion. Consequently, \cref{alg:2} produces the expected output.
\end{proof}

\begin{remark}
Because the coefficients $\ca_k$ and $\cb_k$ are positive and sum to $1$, \cref{alg:1} is numerically stable. Hence, the computation of the matrix $\C$ will be accurate as long as the vector $\vect{a}{}$ is known to sufficient precision. In practice this means that we require accurate and stable evaluation of Bernstein functions and their higher-order derivatives at the breakpoints.
\end{remark}

\begin{remark}
For polynomial B-splines of non-uniform degree, so-called \highlight{multi-degree B-splines}, \Bezier extraction has been analyzed and successfully applied in \cite{Toshniwal:2017a,Toshniwal:2018b}.
An efficient \textsc{Matlab} toolbox implementation illustrating \cref{alg:1,alg:2} can be found in \cite{Speleers:2018}.
\end{remark}

\begin{remark}\label{rmk:abstract-basis}
Whenever the space $\splSpacerp$ admits a basis $\{\BS_k, \; k=1,\ldots, \N\}$ with the properties listed in \cref{thm:B-splines,thm:knotinsertion}, \cref{alg:2} (using \cref{alg:1}) can be applied for an efficient evaluation of these basis functions. In other words, the algorithm does not require that the space $\splSpacerp$ is identified by an admissible weight system (see \cref{weight:assumption}).
\end{remark}


%% file: algorithms/algorithm1.tex
\begin{algorithm}[t!]
  \caption{Nullspace computation of a smoothness constraint based on knot insertion
  \label{alg:1}}{
  \begin{algorithmic}[1]
    \Statex
    \Function{nullspace}{$\vect{a}{} \in \RR^\N$}
	\Let{$\C$}{zero matrix (size: $(\N-1)\times\N$)}
	\Let{$k$}{$1$}
	\While{$k<\N \And \vect{a}{}(k)=0$}		\Comment{\Rd Sub-Matrix $\mat{I}_A$ in \eqref{eq:extract1} \B}
	  \Let{$\C(k,k)$}{$1$}
	  \Let{$k$}{$k+1$}
	\EndWhile
	\Let{$\C(k,k)$}{$1$}
	\While{$k+1<\N \And \vect{a}{}(k+1)\neq0$}		\Comment{\Rd Sub-Matrix $\hat{\mat{C}}$ in \eqref{eq:extract1} \B}
      \Let{$\C(k,k+1)$}{$-\, \Rd\C(k,k) \B \cdot \vect{a}{}(k)\, /\, \vect{a}{}(k+1)$}
      \Let{$\C(k+1,k+1)$}{$1-\C(k,k+1)$}
	  \Let{$k$}{$k+1$}
	\EndWhile
	\While{$k<\N$}				\Comment{\Rd Sub-Matrix $\mat{I}_B$ in \eqref{eq:extract1} \B}
	  \Let{$\C(k,k+1)$}{$1$}
	  \Let{$k$}{$k+1$}
	\EndWhile
	\State \textbf{return}  \C
	\EndFunction
  \end{algorithmic}}
\end{algorithm}

%% file: algorithms/algorithm2.tex
\begin{algorithm}[t!]
  \caption{Generalized \Bezier extraction
    \label{alg:2}}
    {
  \begin{algorithmic}[1]
    \Statex
    \Function{extraction\_operator}{$\AA \in \RR^{\M\times\O}$}
	\Let{$\CC$}{identity matrix (size: $\M\times\M$)}					\Comment{Initialize extraction operator}
	\For{$\cons = 1:\O$}						\Comment{Loop over smoothness constraints}
		\Let{$\C$}{$\Call{nullspace}{\AA(:,\cons)}$}          	\Comment{Compute nullspace of $\cons$-th column of \AA}
		\Let{$\CC$}{$\C \ast \CC$}    			\Comment{update \CC}
		\Let{$\AA$}{$ \C \ast \AA$}				\Comment{update \AA}
	\EndFor
	\State \textbf{return}  \CC
	\EndFunction
    \Statex
 \end{algorithmic}}
\end{algorithm}

%% file: sections/section6_gb.tex
\section{Case study: generalized polynomial B-splines} \label{sec:6_gb}
In this section we consider a special class of GTB-splines, the so-called \highlight{generalized polynomial B-splines (GPB-splines)}; see \cite{ManniRS:2017} and references therein. They can be seen as the minimal extension of polynomial splines of non-uniform degree still offering a wide variety of additional flexibility.

Given $\p_i\geq2$, let $\mathfrak{u}^{(i)},\mathfrak{v}^{(i)}\in C^{\p_i}([\x_{i-1}, \x_{i}])$, and
\begin{equation*}
U^{(i)}:=D^{\p_i-1}\mathfrak{u}^{(i)}, \quad V^{(i)}:=D^{\p_i-1}\mathfrak{v}^{(i)},
\end{equation*}
such that $\GP{}{(i)}:=\Span{U^{(i)},V^{(i)}}$ is an ECT-space on $[\x_{i-1}, \x_{i}]$.
There exists a unique couple of functions $\normalize{U}^{(i)},\normalize{V}^{(i)}\in \GP{}{(i)}$ such that
\begin{equation*}
\normalize{U}^{(i)}(\x_{i-1})=1,\quad \normalize{U}^{(i)}(\x_{i})=0, \quad
\quad \normalize{V}^{(i)}(\x_{i-1})=0,\quad \normalize{V}^{(i)}(\x_{i})=1.
\end{equation*}
The generalized polynomial space of degree $\p_i\geq2$ on the closed interval $[\x_{i-1}, \x_{i}]$ is then defined by
\begin{equation} \label{eq:gb-space}
\GP{\p_i}{(i)} := \Span{1, \x, \x^2, \ldots, \x^{\p_i-2},\mathfrak{u}^{(i)}(\x), \mathfrak{v}^{(i)}(\x)}.
\end{equation}
We refer to \cite{Costantini:2005} for more details on such spaces and their properties.

\begin{example}\label{ex:examples_space}
Popular choices for $\mathfrak{u}^{(i)}$ and $\mathfrak{v}^{(i)}$ are given by
\begin{itemize}
\item   $\mathfrak{u}^{(i)}(\x) = \x^{\p_i-1}, \quad \mathfrak{v}^{(i)}(\x) = \x^{\p_i}$,
\item   $\mathfrak{u}^{(i)}(\x) = \sinh(\omega\x), \quad \mathfrak{v}^{(i)}(\x) = \cosh(\omega\x), \quad 0<\omega$,
\item   $\mathfrak{u}^{(i)}(\x) = \sin(\omega\x), \quad \mathfrak{v}^{(i)}(\x) = \cos(\omega\x), \quad 0<\omega(\x_{i}-\x_{i-1})<\pi$,
\end{itemize}
which correspond to the classical polynomial, exponential and trigonometric spaces, respectively.
\end{example}

\begin{remark} Once $\mathfrak{u}^{(i)}$ and $\mathfrak{v}^{(i)}$ are chosen they uniquely define $U^{(i)}$ and $V^{(i)}$. Conversely, if $U^{(i)}$ and $V^{(i)}$ are chosen they uniquely define the space $\GP{\p_i}{(i)}$.
\end{remark}

\begin{remark} \label{rmk:gb-weights}
  From \cite[Example~10]{Lyche:2019} we know that the space $\GP{\p_i}{(i)}$ is an ECT-space generated by the weights
  \begin{align*}
    &w^{(i)}_0(\x) =\cdots=w^{(i)}_{\p_i-2}(\x)=1, \\
    &w^{(i)}_{\p_i-1}(\x) = \normalize{U}^{(i)}(\x)+\normalize{V}^{(i)}(\x), \\
    &w^{(i)}_{\p_i}(\x) = \frac{\normalize{U}^{(i)}(\x)D\normalize{V}^{(i)}(\x)-\normalize{V}^{(i)}(\x)D\normalize{U}^{(i)}(\x)}{\big(\normalize{U}^{(i)}(\x)+\normalize{V}^{(i)}(\x)\big)^2}.
  \end{align*}
  It can be easily checked that these local weights are admissible for the spline space $\splSpacerp$, where the different pieces belong to $\GP{\p_i}{(i)}$, $i=1, \ldots, m$ (see \cref{weight:assumption}), whenever $\r_i<\min\{\p_i,\p_{i+1}\}$. Moreover, they fulfill the partition-of-unity assumption in \cref{eq:weights-pou}.
\end{remark}

From \cref{rmk:gb-weights,thm:B-splines} it follows that there exist GTB-splines for spline spaces composed of generalized polynomial spaces as in \cref{eq:gb-space} of possibly different dimensions. These GTB-splines can be computed by the algorithmic procedure described in \cref{sec:5} starting from the local Bernstein bases. In the remainder of the section we discuss and illustrate the Bernstein basis in case of generalized polynomial spaces.

Using the explicit expressions of the weights provided in \cref{rmk:gb-weights}, we can simplify the recurrence relation in \cref{eq:rec-bern-0}--\cref{eq:rec-bern-q} of the local Bernstein functions $\bs^{(i)}_{j} := \bs_{j,\p_{i}}$, $j=0, \ldots, \p_i$ as follows; see also \cite[Section 4]{Lyche:2019}. For $\q=1,\ldots, \p_i$ and $j=0,\ldots,\q$, the function $\bs_{j,\q}$ can be evaluated at $\x \in [\x_{i-1}, \x_{i}]$ as
\begin{equation} \label{eq:rec-bern-gb-1}
\bs_{0,1}(\x) := \normalize{U}^{(i)}(\x), \quad \bs_{1,1}(\x) := \normalize{V}^{(i)}(\x),
\end{equation}
and
\begin{equation} \label{eq:rec-bern-gb-q}
\bs_{j,\q}(\x) := \begin{dcases}
1-\int_{\x_{i-1}}^\x \dfrac{\bs_{0,\q-1}(\y)}{\bbint_{0,\q-1}}\dint\y, & j=0,
\\
\int_{\x_{i-1}}^\x \biggl[\dfrac{\bs_{j-1,\q-1}(\y)}{\bbint_{j-1,\q-1}} - \dfrac{\bs_{j,\q-1}(\y)}{\bbint_{j,\q-1}}\biggr]\dint\y, &  0< j < q,
\\
\int_{\x_{i-1}}^\x \dfrac{\bs_{\q-1,\q-1}(\y)}{\bbint_{\q-1,\q-1}}\dint\y, & j=\q,
\end{dcases}\quad \q>1,
\end{equation}
where $\bbint_{j, \q-1}$ is defined in \cref{eq:bern-integral}.

\begin{example}\label{ex:gb-spaces-pol}
The classical (polynomial) Bernstein basis of degree $\p_i=\q$ on $[\x_{i-1}, \x_{i}]=[0,1]$ can be expressed as
\begin{equation*}  
\bs_{j,\q}(\x)=\binom{\q}{j}\x^{j}(1-\x)^{\q-j},
\quad j=0, \ldots, \q.
\end{equation*}
\end{example}

\begin{example}\label{ex:gb-spaces-exp}
The generalized Bernstein basis for the exponential space in \cref{ex:examples_space} defined on $[\x_{i-1}, \x_{i}]=[0,1]$ reads for degree $\q=1$,
\begin{align*}  
  \bs_{0,1}(\x) = \frac{\sinh(\omega(1-\x))}{\sinh(\omega)}, \quad
  \bs_{1,1}(\x) = \frac{\sinh(\omega \x)}{\sinh(\omega)},
\end{align*}
and for degree $\q=2$,
\begin{gather*}
  \bs_{0,2}(\x) = \frac{1-\cosh(\omega(1-\x))}{1-\cosh(\omega)}, \quad
  \bs_{2,2}(\x) = \frac{1-\cosh(\omega \x)}{1-\cosh(\omega)}, \\
  \bs_{1,2}(\x) = \frac{\cosh(\omega(1-\x))+\cosh(\omega \x)-\cosh(\omega)-1}{1-\cosh(\omega)}.
\end{gather*}
\end{example}

\begin{example}\label{ex:gb-spaces-trig}
The Bernstein basis for the trigonometric space in \cref{ex:examples_space} defined on $[\x_{i-1}, \x_{i}]=[0,1]$ reads for degree $\q=1$,
\begin{align*} 
  \bs_{0,1}(\x) = \frac{\sin(\omega(1-\x))}{\sin(\omega)}, \quad
  \bs_{1,1}(\x) = \frac{\sin(\omega \x)}{\sin(\omega)},
\end{align*}
and for degree $\q=2$,
\begin{gather*}
  \bs_{0,2}(\x) = \frac{1-\cos(\omega(1-\x))}{1-\cos(\omega)}, \quad 
  \bs_{2,2}(\x) = \frac{1-\cos(\omega \x)}{1-\cos(\omega)},\\
  \bs_{1,2}(\x) = \frac{\cos(\omega(1-\x))+\cos(\omega \x)-\cos(\omega)-1}{1-\cos(\omega)}.
\end{gather*}
\end{example}

\begin{remark}
Instead of using the recurrence relation \cref{eq:rec-bern-gb-1}--\cref{eq:rec-bern-gb-q}, each Bernstein function $\bs^{(i)}_{j}$ can also be computed by solving in the space $\GP{\p_i}{(i)}$ the  Hermite interpolation problem stated in \cref{rmk:Hermite-Bernstein}.
\end{remark}

GT-spline spaces with pieces drawn from (different) generalized polynomial spaces containing polynomial, exponential or trigonometric functions (see, e.g., \cref{ex:gb-spaces-pol,ex:gb-spaces-trig}) are of particular interest both in geometric design and numerical simulation because they offer a valid alternative to NURBS. Indeed, they allow for a locally exact representation of conic sections with respect to (almost) arc length, and moreover, the derivative  spaces belong to the same class, exactly as for polynomial splines; see \cite{ManniRS:2017} and references therein for further details.

%% file: sections/section6_c.tex
\section{Numerical examples} \label{sec:6}
In this section we present two numerical examples to illustrate the algorithmic procedure in \cref{sec:5} and a simple application of GTB-splines for exact smooth representation of profiles containing conic section segments.

\begin{example} \label{ex:1}
Consider a GT-spline space $\splSpacerp$ defined by
\begin{equation*}
\domain = \{0,1,5/2,5\}, \quad
\mbf{\p}{} = \{2,3,4\}, \quad
\mbf{\r}{} = \{-1,2,2,-1\},
\end{equation*}
and
\begin{gather*}
\ECT{2}{(1)}=\Span{1, \x, \x^2}, 
\quad
\ECT{3}{(2)}=\Span{1, \x, \cos(\pi\x/2), \sin(\pi\x/2)}, \\
\ECT{4}{(3)}=\Span{1, \x, \x^2, \sinh(10\x), \cosh(10\x)}.
\end{gather*}
All these spaces are special instances of generalized polynomial spaces discussed in \cref{sec:6_gb}. Sequences of admissible weights for $\splSpacerp$ can be computed as described in \cref{rmk:gb-weights} by taking into account the explicit expressions for $\normalize{U}^{(i)}$ and $ \normalize{V}^{(i)}$, $i=1,2,3$:
\begin{alignat*}{3}
  &\normalize{U}^{(1)}(x)=1-x, \quad 
  &&\normalize{U}^{(2)}=-\sqrt{2}\cos(\pi/4+\pi x/2), \quad
  &&\normalize{U}^{(3)}=\frac{\sinh(50-10x)}{\sinh(25)}, \\ 
  &\normalize{V}^{(1)}(x)=x, \quad
  &&\normalize{V}^{(2)}=-\sqrt{2}\cos(\pi x/2), \quad
  &&\normalize{V}^{(3)}=-\frac{\sinh(25-10x)}{\sinh(25)}. 
\end{alignat*}
The resulting GT-spline space $\splSpacerp$ has dimension $6$ and, in view of \cref{thm:B-splines}, possesses a GTB-spline basis. For this space, the knot vectors, $\mbf{\lknot}$ and $\mbf{\rknot}$, and the start-point and end-point smoothness, $\lsmooth$ and $\rsmooth$, are depicted in \cref{tab:ex1}.
\cref{fig:ex1} shows the spline functions constructed at different iterations in \cref{alg:2}, together with their first and second derivatives. Starting from the 12 global Bernstein basis functions of $\splSpacep$ (\cref{fig:ex1}, first row), \cref{alg:2} incrementally increases the smoothness at the breakpoints  until  the $6$ final GTB-splines (\cref{fig:ex1}, last row) are obtained.
\end{example}

\begin{table}[t!]
 \centering
\caption{The triples $\left([\lknot_k, \rknot_k],\lsmooth(k),\rsmooth(k)\right)$ defining $\BS_k\in \splSpacerp$ for $k=1,\ldots,\N$ as in \cref{ex:1}.}
\label{tab:ex1}
 \begin{tabular}{|c|r|r|r|r|r|r|}
 \hline
 $k$           & 1   & 2 & 3 & 4 & 5   & 6    \\ \hline  \hline
 $\lknot_k$    & 0   & 0 & 0 & 1 & 5/2 & 5/2  \\
 $\rknot_k$    & 5/2 & 5 & 5 & 5 & 5   & 5    \\ \hline \hline
 $\lsmooth(k)$ & -1  & 0 & 1 & 2 & 2   & 3    \\ \hline \hline
 $\rsmooth(k)$ & 2   & 3 & 2 & 1 & 0   & -1   \\ \hline
 \end{tabular}
\end{table}

\begin{figure}[t!]
\centering
\subfigure[$\mbf{\r}{} = \{-1,-1,-1,-1\}$]
{\includegraphics[height=3cm]{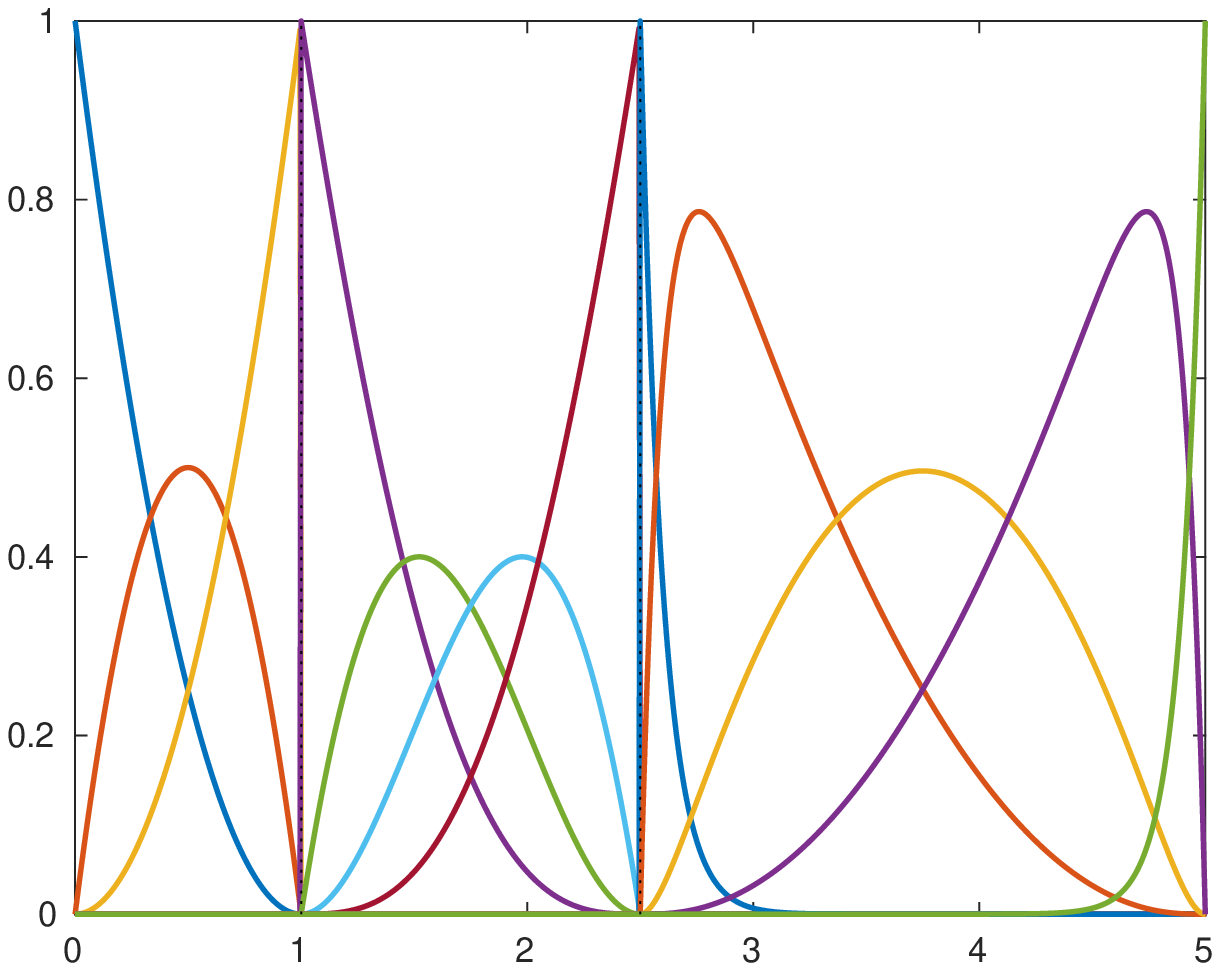} 
\hspace*{0.2cm}
\includegraphics[height=3cm]{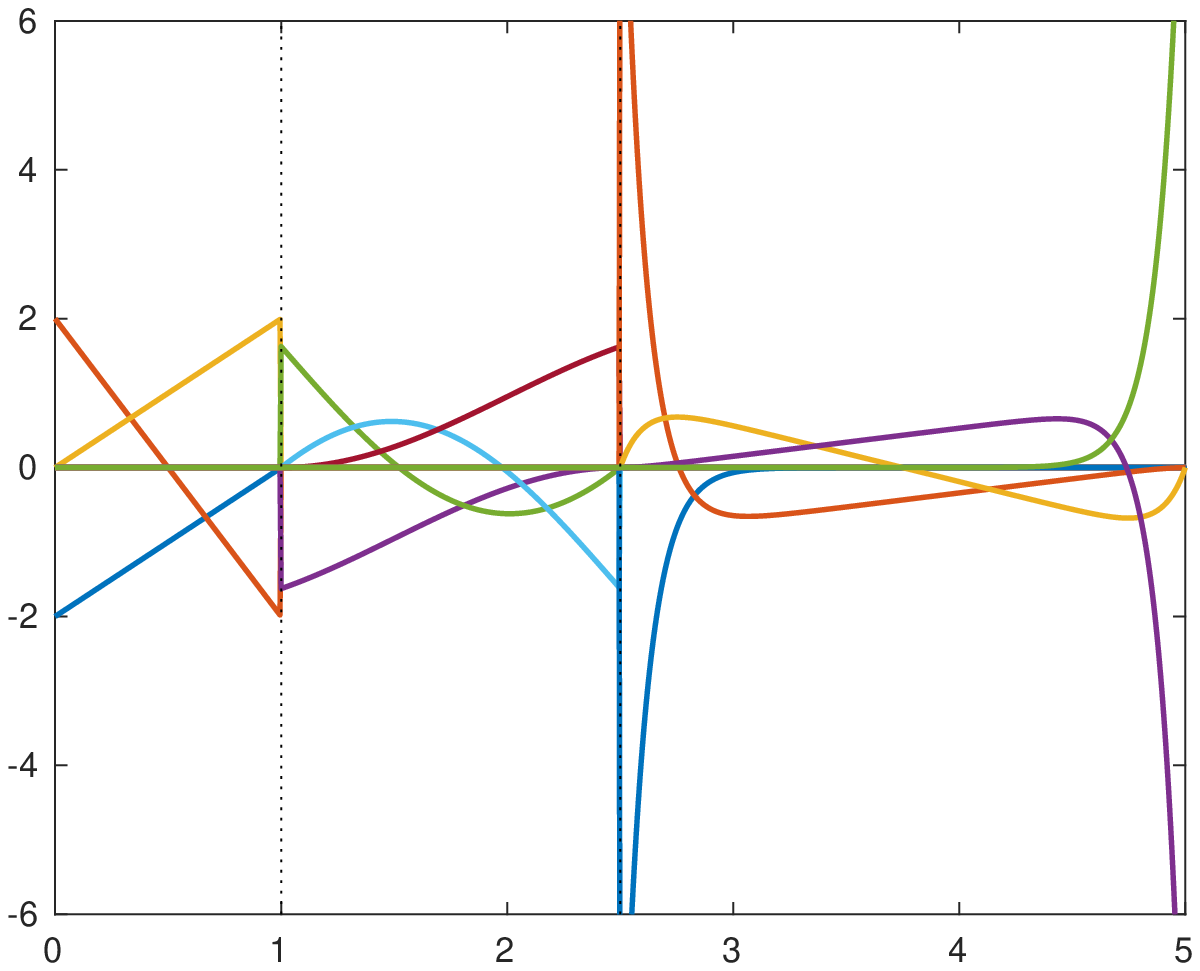}
\hspace*{0.2cm}
\includegraphics[height=3cm]{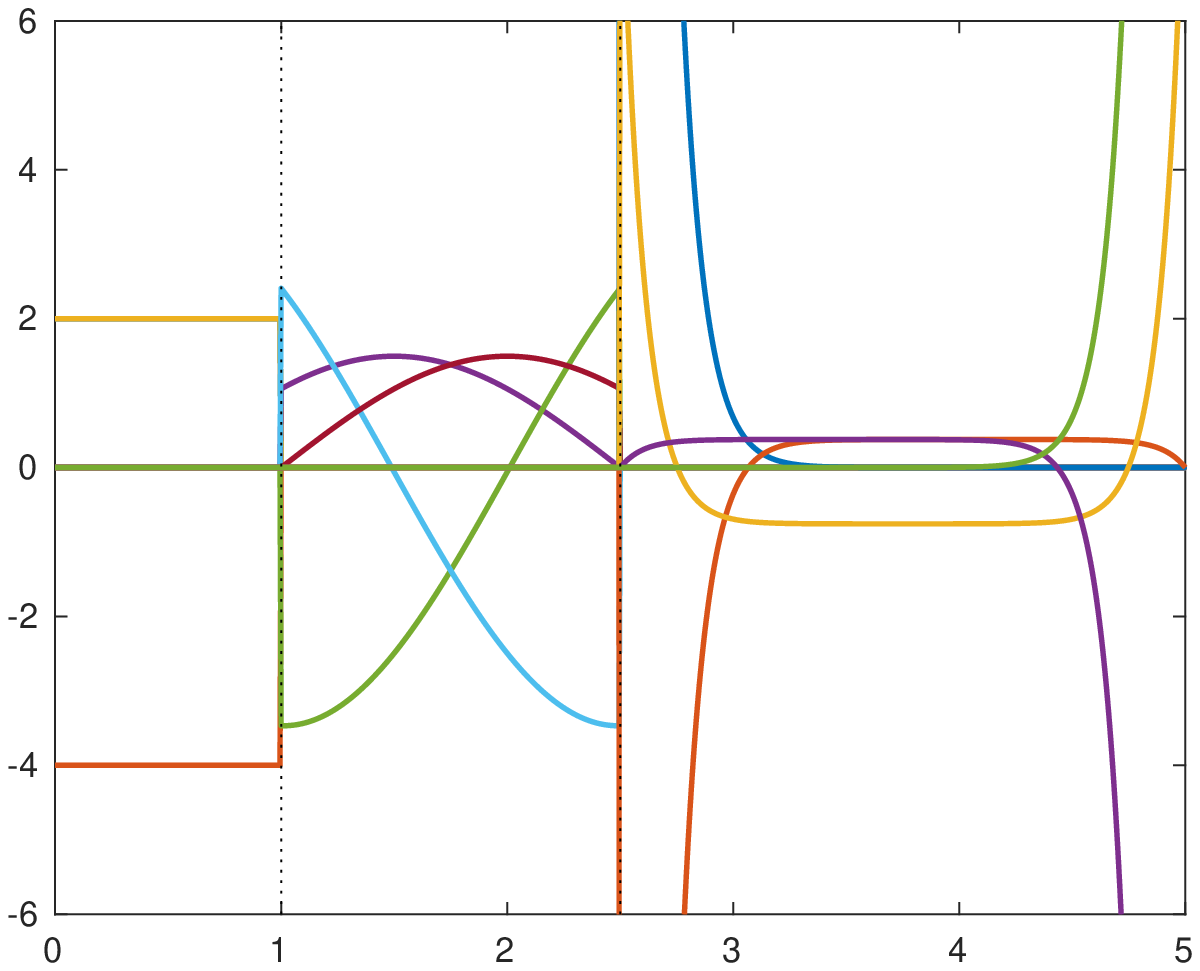}}
\subfigure[$\mbf{\r}{} = \{-1,0,0,-1\}$]
{\includegraphics[height=3cm]{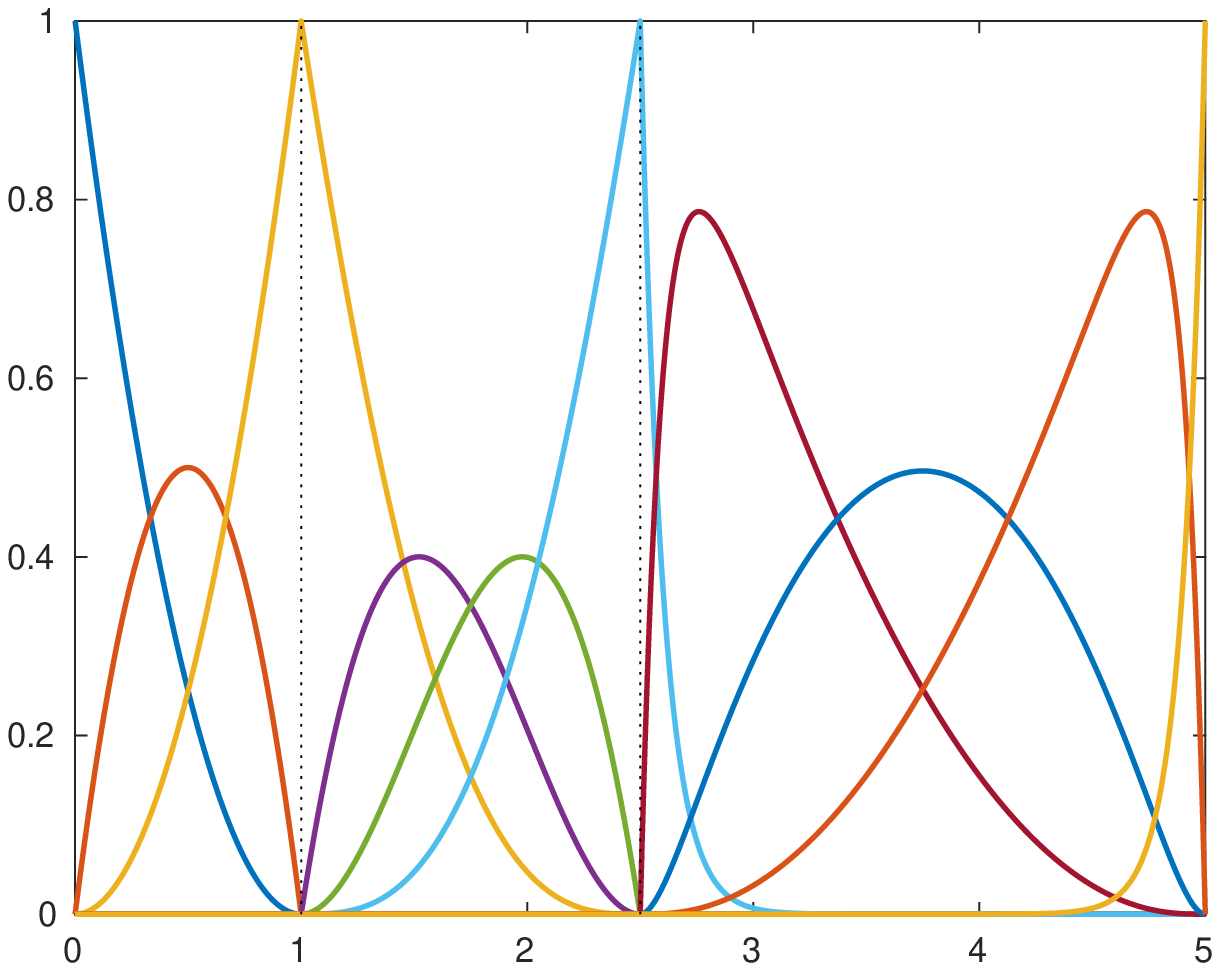}
\hspace*{0.2cm}
\includegraphics[height=3cm]{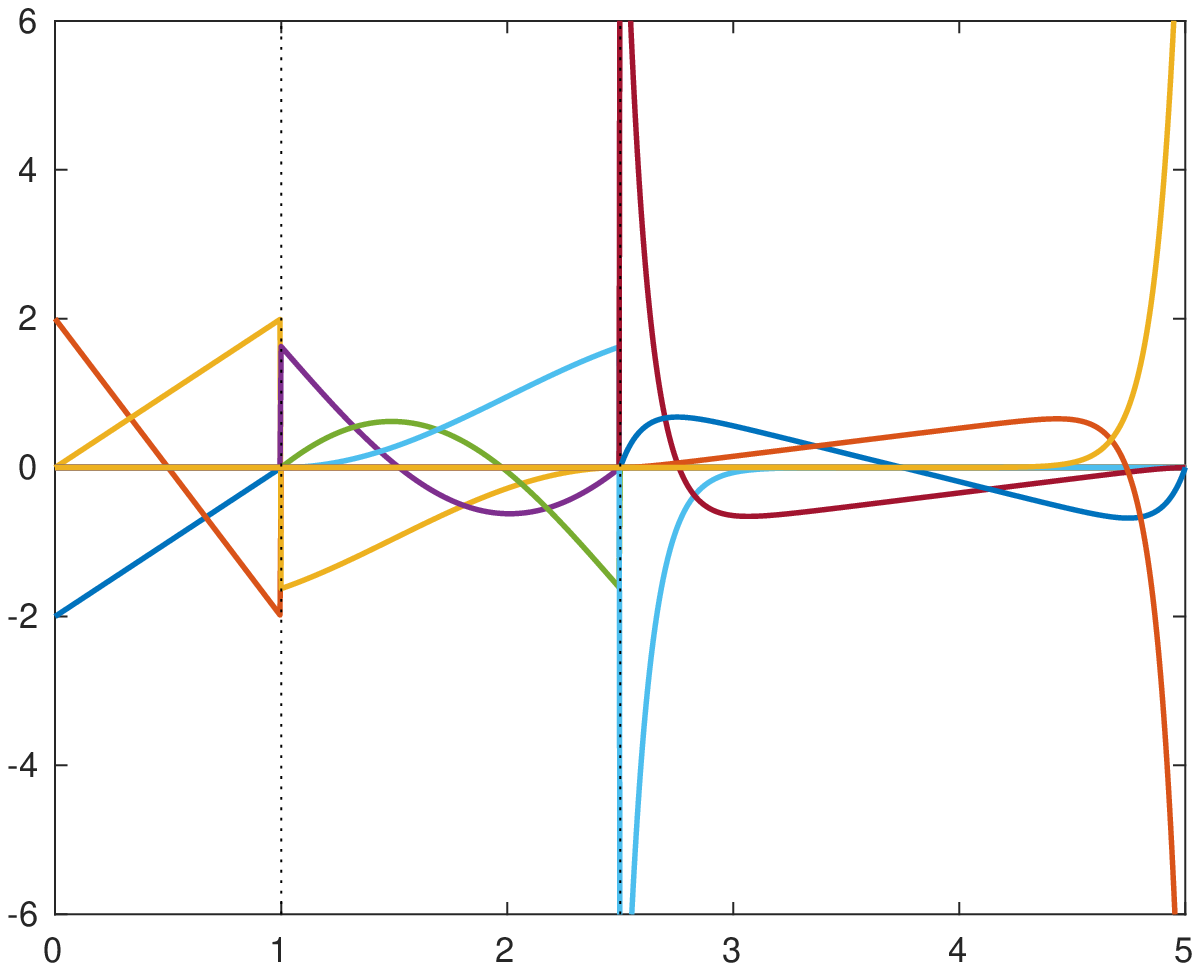}
\hspace*{0.2cm}
\includegraphics[height=3cm]{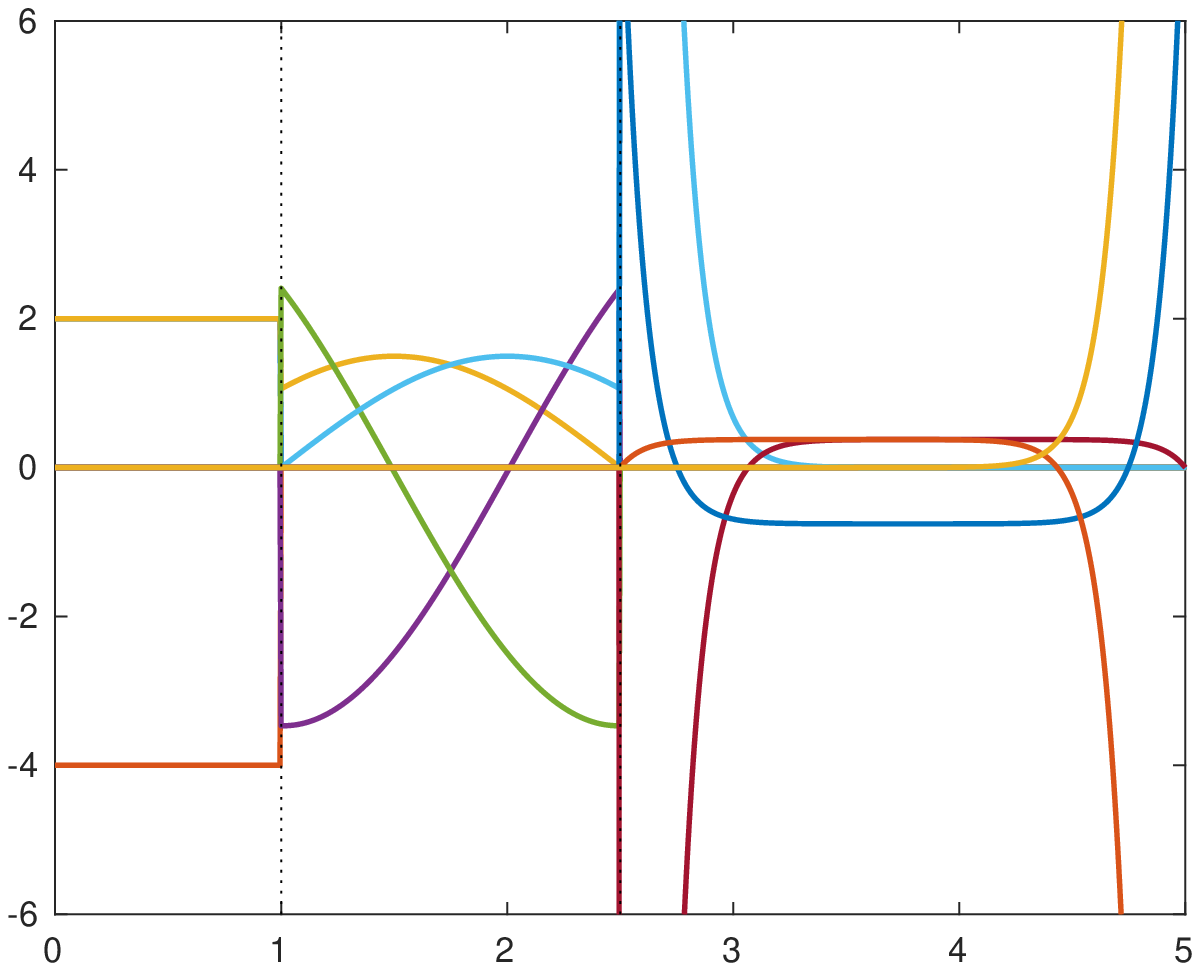}}
\subfigure[$\mbf{\r}{} = \{-1,1,1,-1\}$]
{\includegraphics[height=3cm]{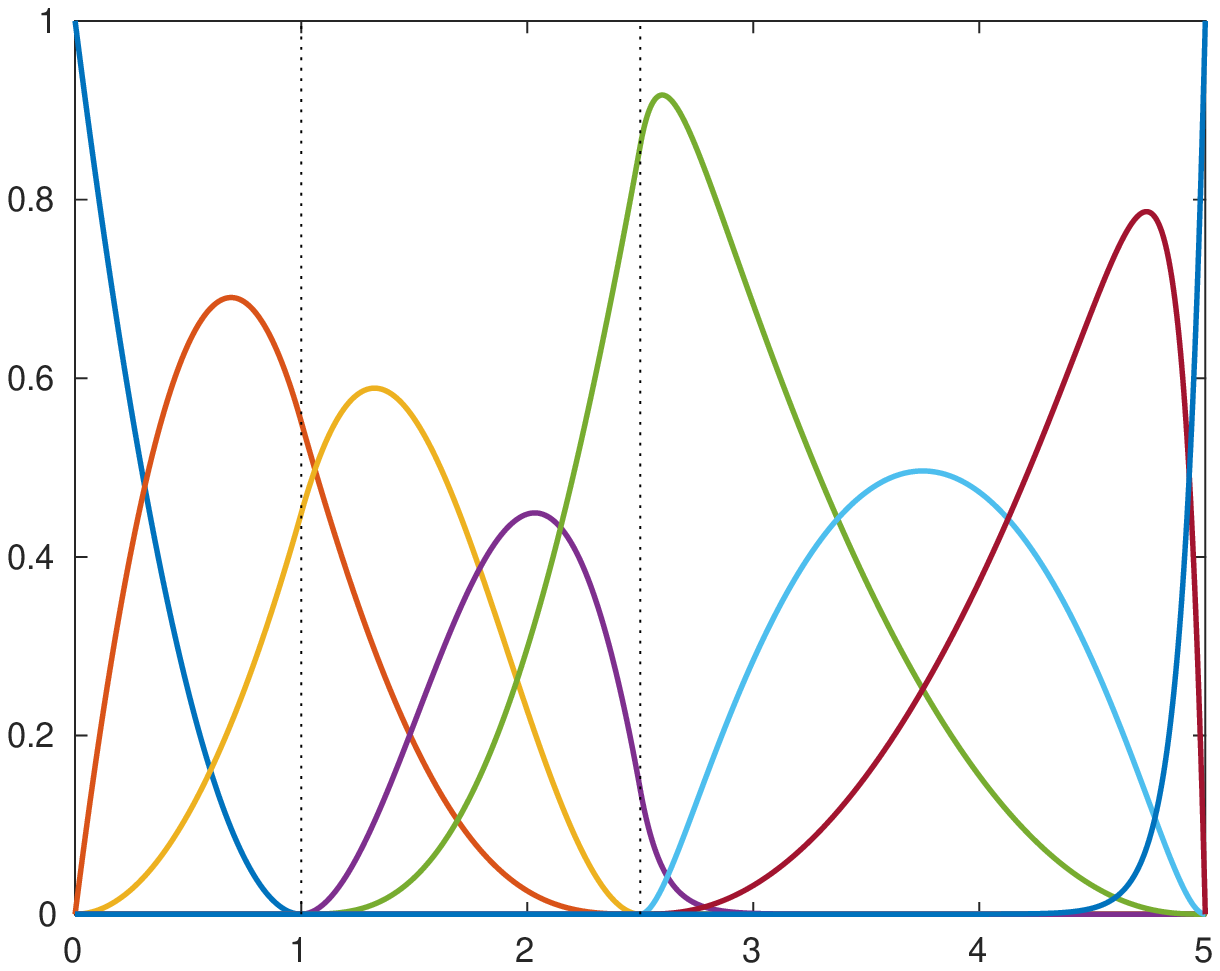}
\hspace*{0.2cm}
\includegraphics[height=3cm]{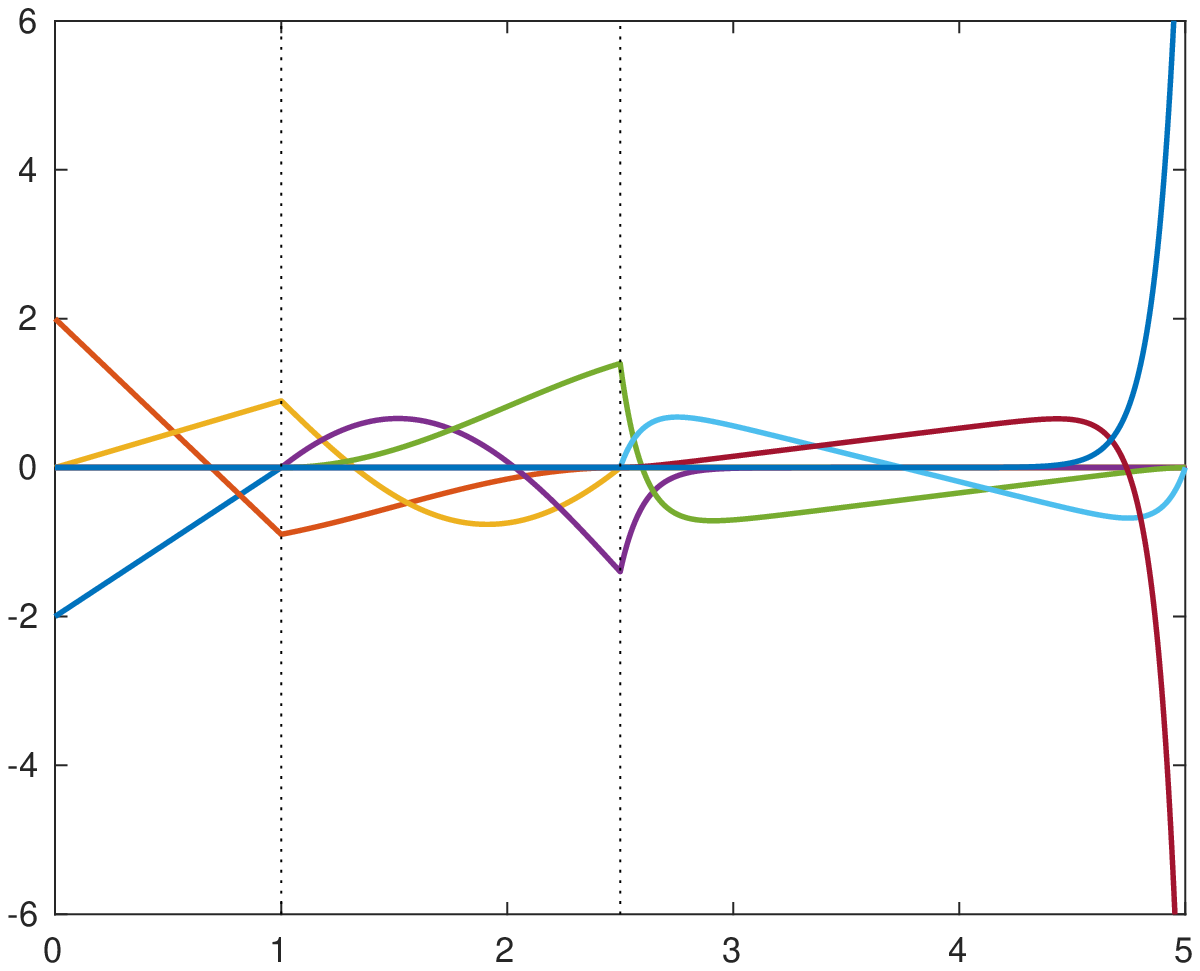}
\hspace*{0.2cm}
\includegraphics[height=3cm]{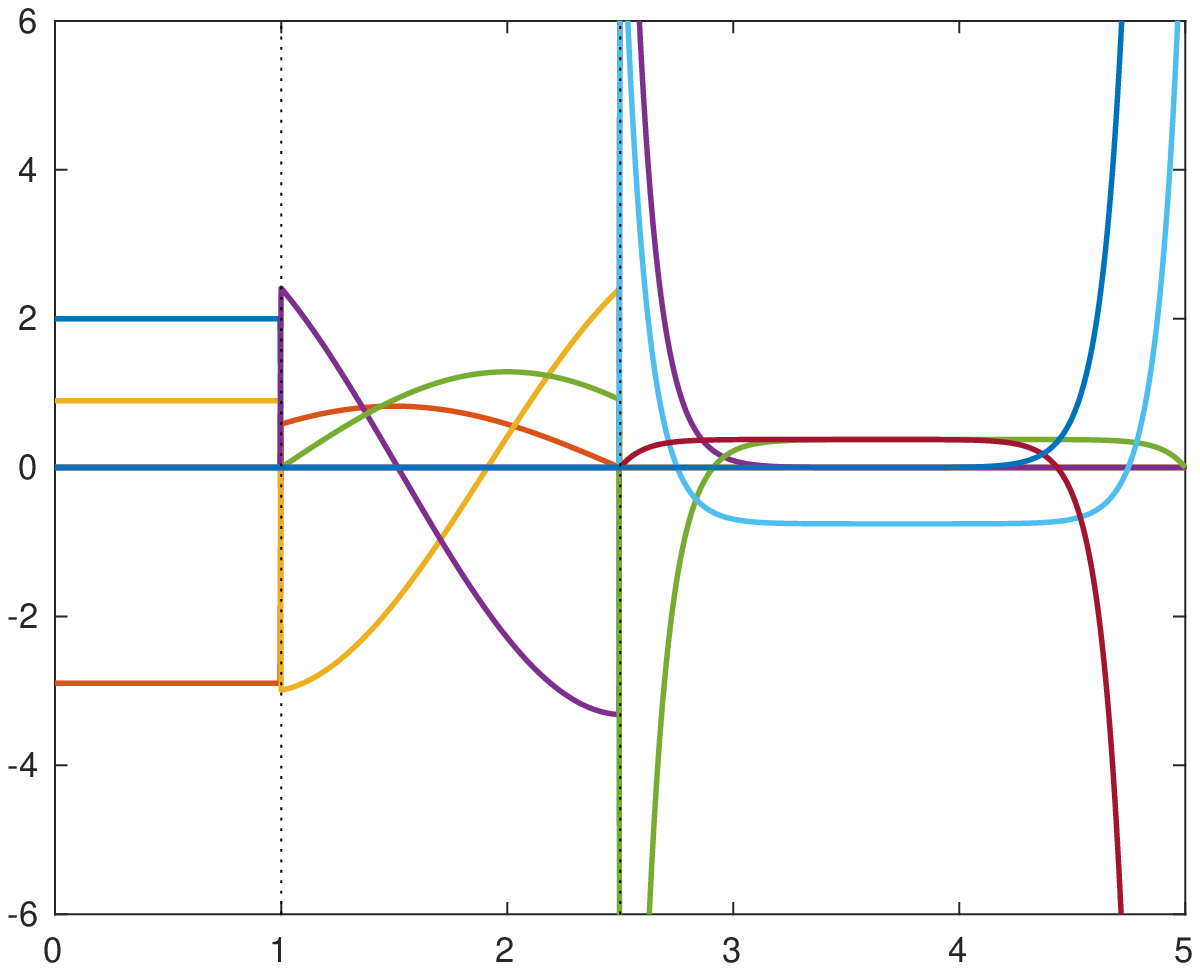}}
\subfigure[$\mbf{\r}{} = \{-1,2,2,-1\}$]{
\includegraphics[height=3cm]{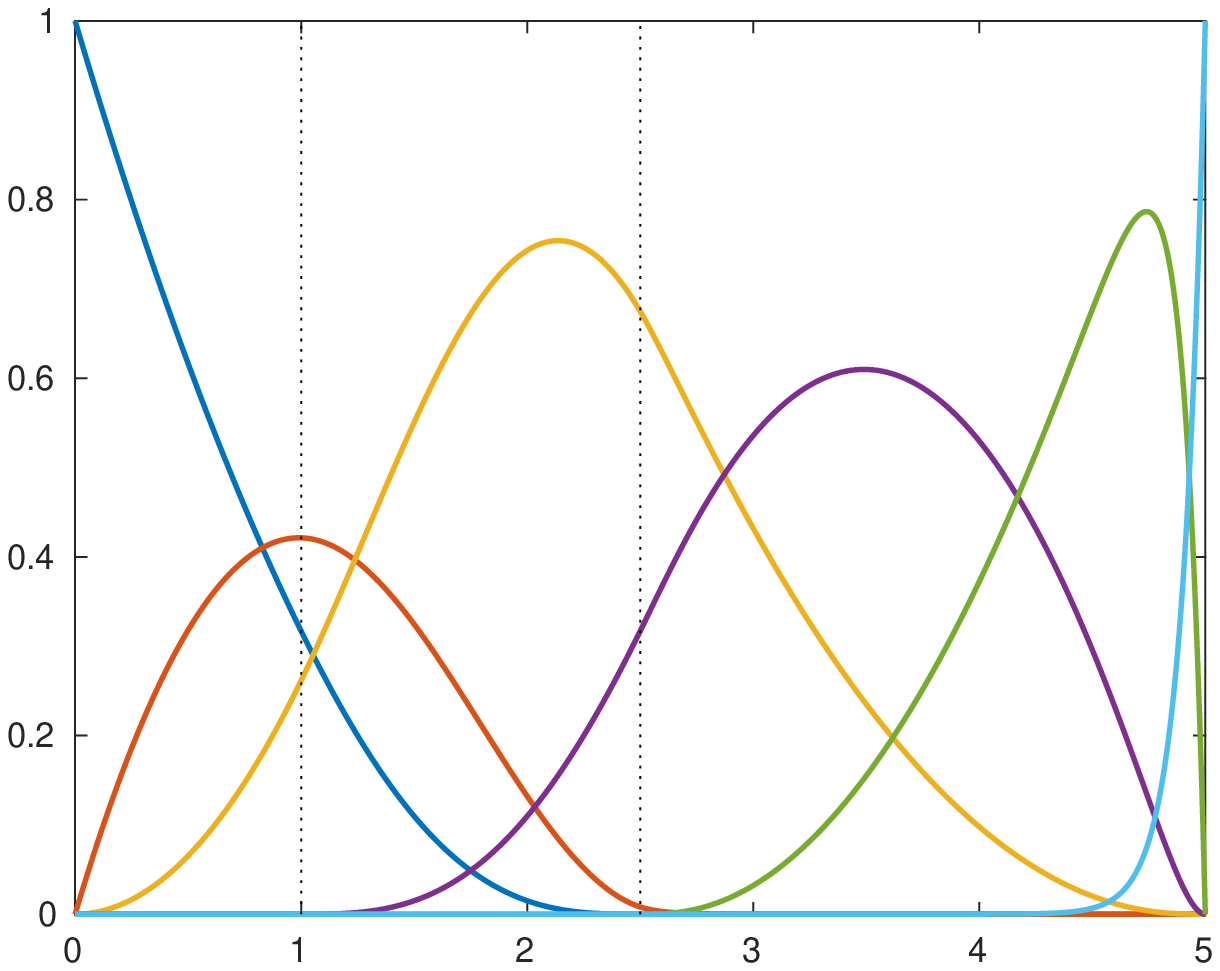}
\hspace*{0.2cm}
\includegraphics[height=3cm]{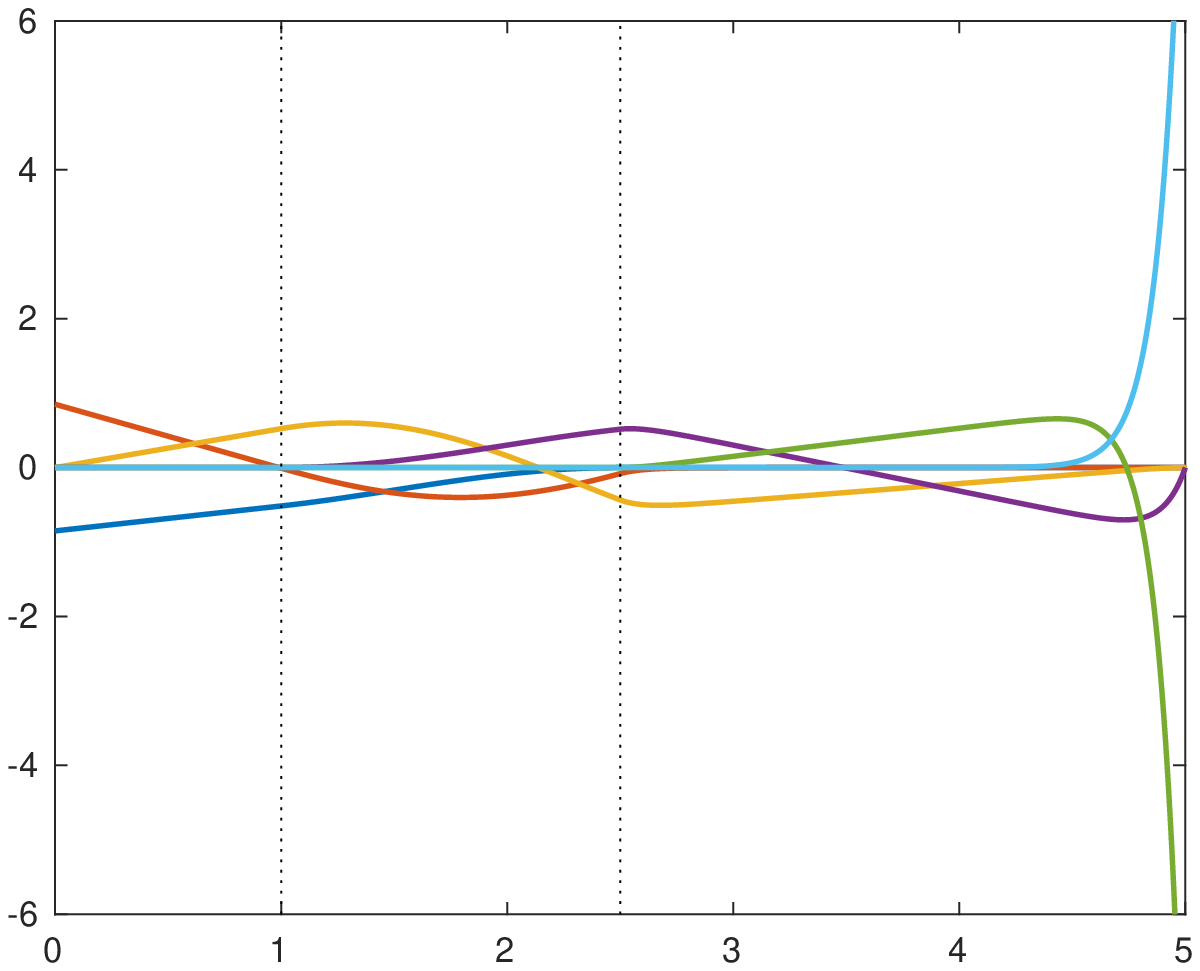}
\hspace*{0.2cm}
\includegraphics[height=3cm]{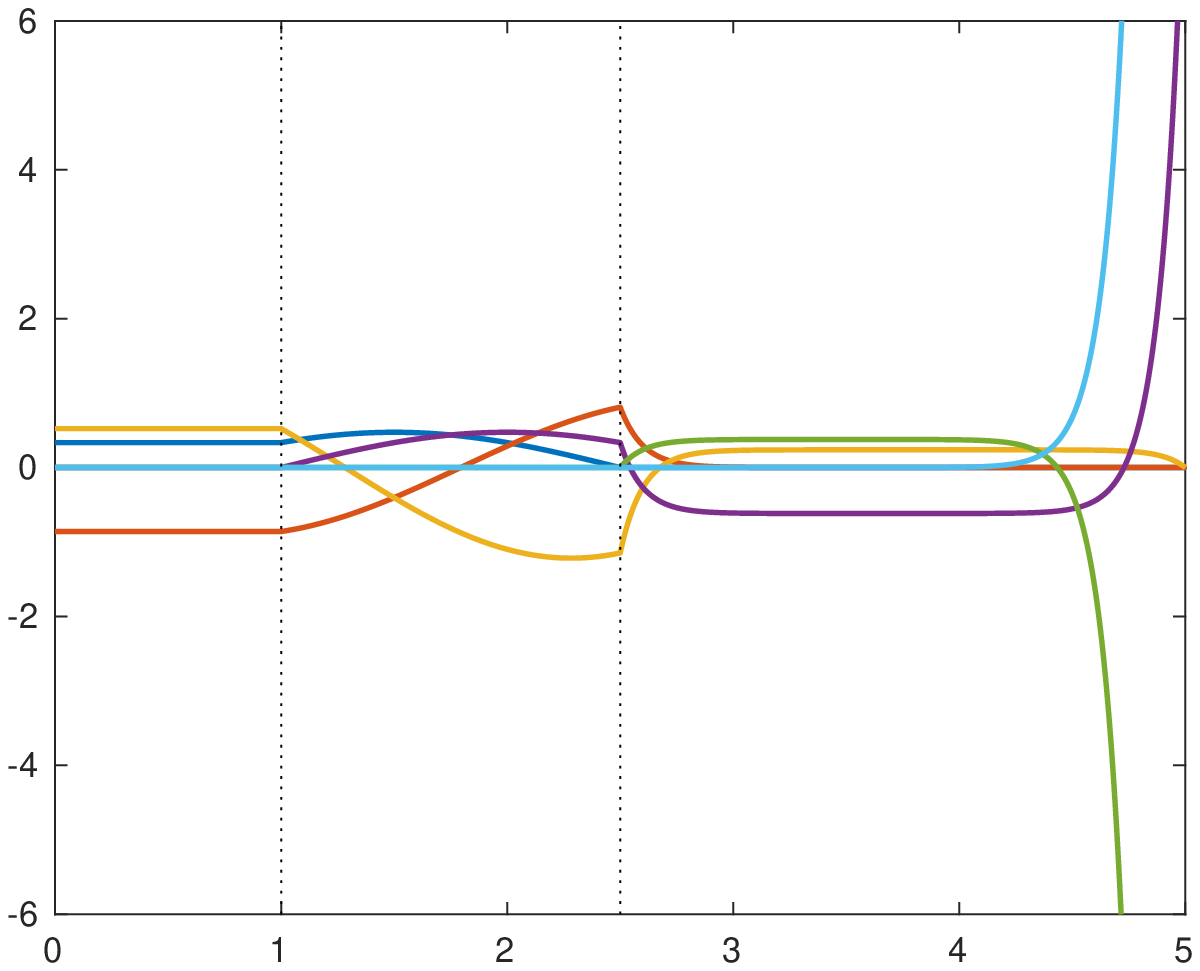}}
\caption{Sets of GTB-splines for the GT-spline spaces $\splSpacerp$ built from the ECT-spaces $\ECT{p_i}{(i)}$, $i=1,2,3$ defined in \cref{ex:1} and different smoothness classes $\mbf{\r}{}$ (left column), together with their first derivatives (middle column) and second derivatives (right column). Knot positions are visualized by vertical dotted lines.}
\label{fig:ex1}
\end{figure}

\begin{example} \label{ex:2}
The profile depicted in \cref{fig:ex2} (left) consists of one circular arc, with center $(2,0)$ and radius $1$, connected by a straight line segment to another circular arc, with center $(0,3)$ and radius $2$. More precisely, we are considering the profile described by the parametric curve
\begin{equation*}
(X(\x),\; Y(\x))=
\begin{cases}
(2-\sin (\x),\; \cos(\x)), & \x\in [-3\pi/4, 0),
\\
(2-\x,\; 1), & \x\in [0, 2),
\\
(-2\sin(\x/2-1),\; 3-2\cos(\x/2-1)), & \x\in [2, 2+\pi].
\end{cases}
\end{equation*}
One can easily verify that this parameterization is $C^1$ in both components.
This profile can be exactly represented as a parametric $C^1$ GT-spline curve whose components belong to the 4-dimensional GT-spline space $\splSpacerp$ defined by
\begin{equation*}
\domain = \{-3\pi/4, 0,2, 2+\pi\}, \quad
\mbf{\p}{} = \{2,1,2\}, \quad
\mbf{\r}{} = \{-1,1,1,-1\},
\end{equation*}
and
\begin{gather*}
\ECT{2}{(1)}=\Span{1, \cos(\x), \sin(\x)}, \quad
\ECT{1}{(2)}=\Span{1, \x}, \\
\ECT{2}{(3)}=\Span{1, \cos(\x/2), \sin(\x/2)}.
\end{gather*}
The parametric coefficients (control points) are given by
\begin{equation*}
  (2+\sqrt{2}/2,\; -\sqrt{2}/2), \quad 
  (3+\sqrt{2},\; 1), \quad  
  (-2,\; 1), \quad 
  (-2,\; 3).
\end{equation*}
The representation in terms of GTB-splines and the corresponding control polygon is visualized in \cref{fig:ex2} (left). We see that the control polygon nicely encapsulates the profile. The $C^1$ basis functions used in the representation are shown in \cref{fig:ex2} (middle), and their first derivatives in \cref{fig:ex2} (right).
\end{example}

\begin{figure}[t!]
\centering
{\includegraphics[height=3cm]{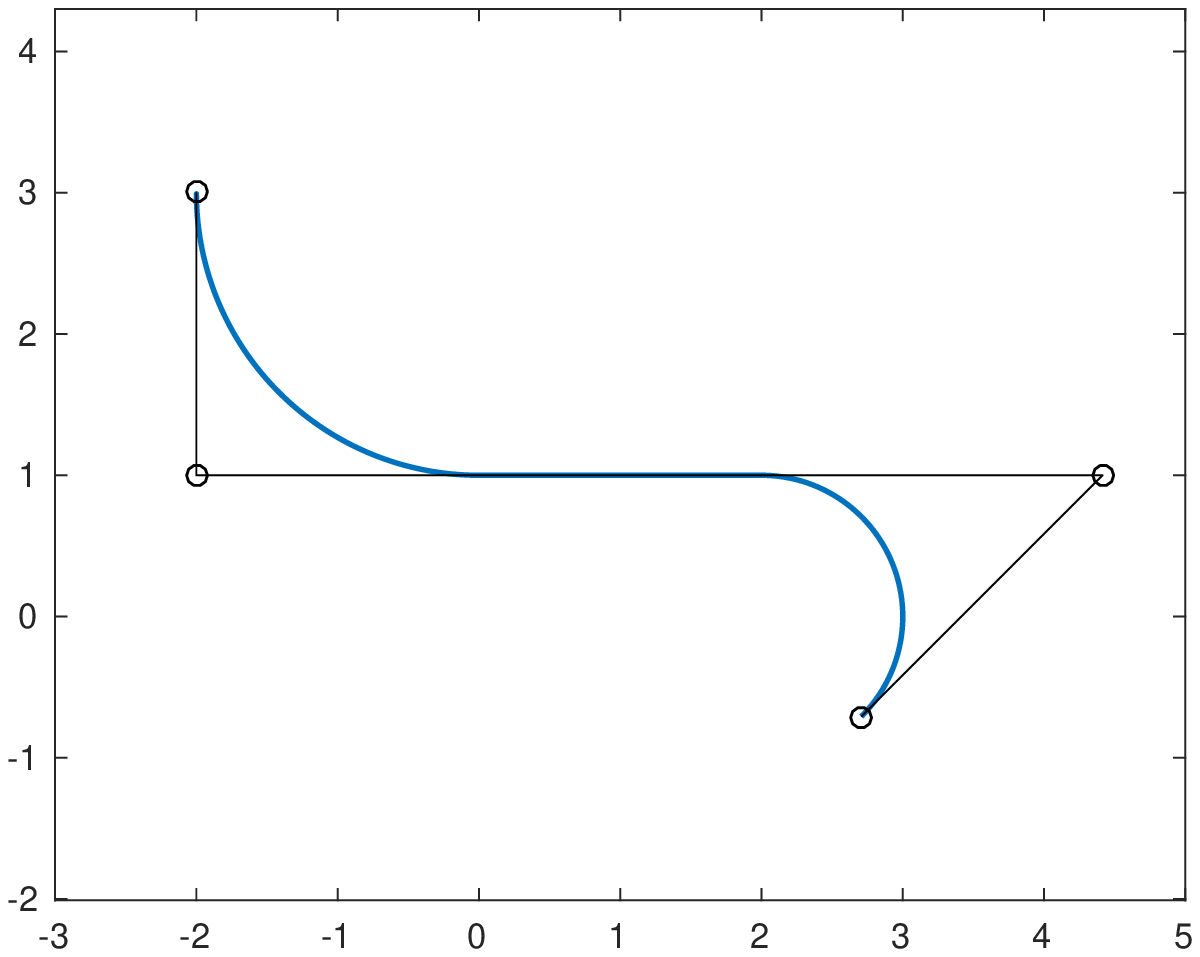}
\hspace*{0.2cm}
\includegraphics[height=3cm]{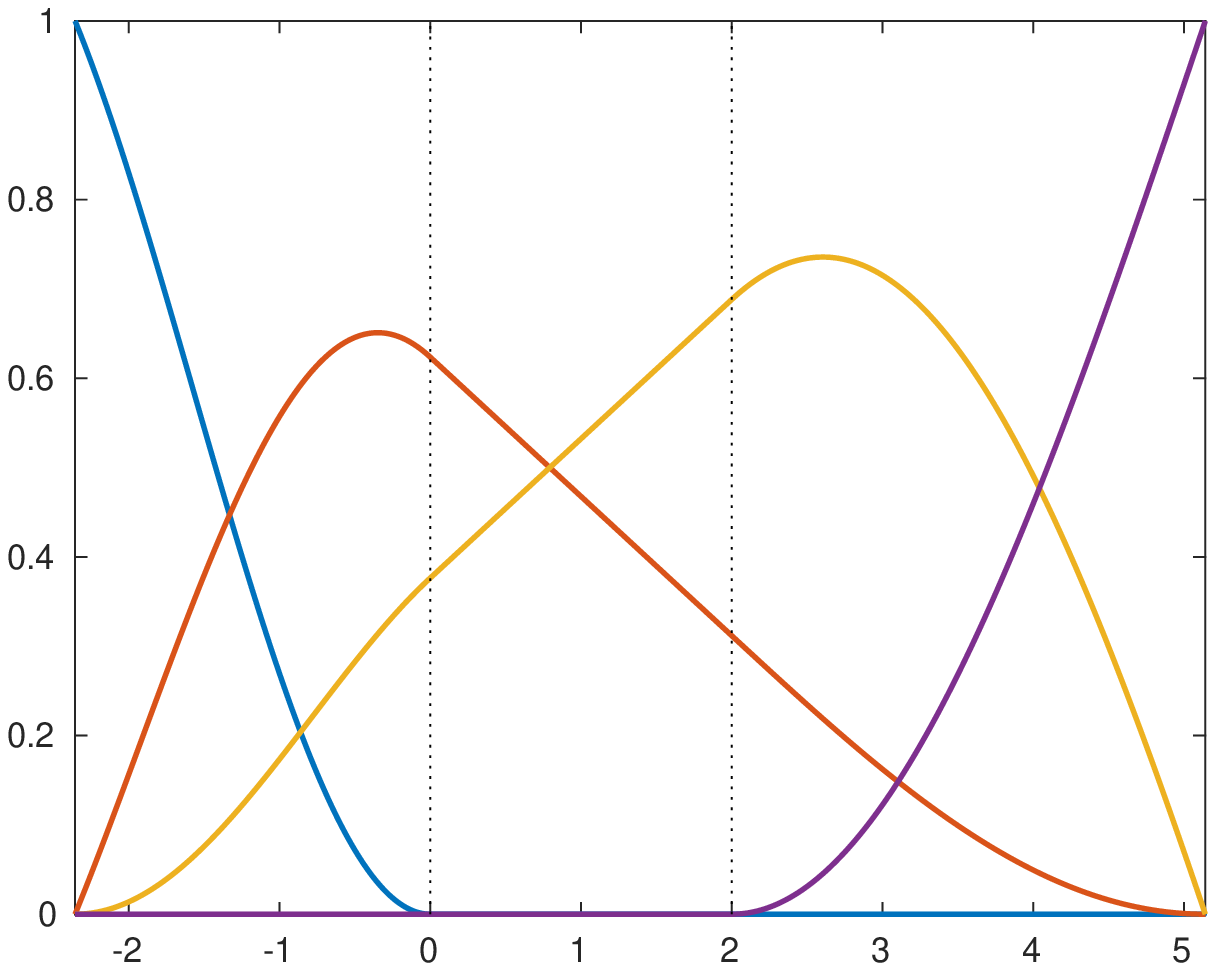} 
\hspace*{0.2cm}
\includegraphics[height=3cm]{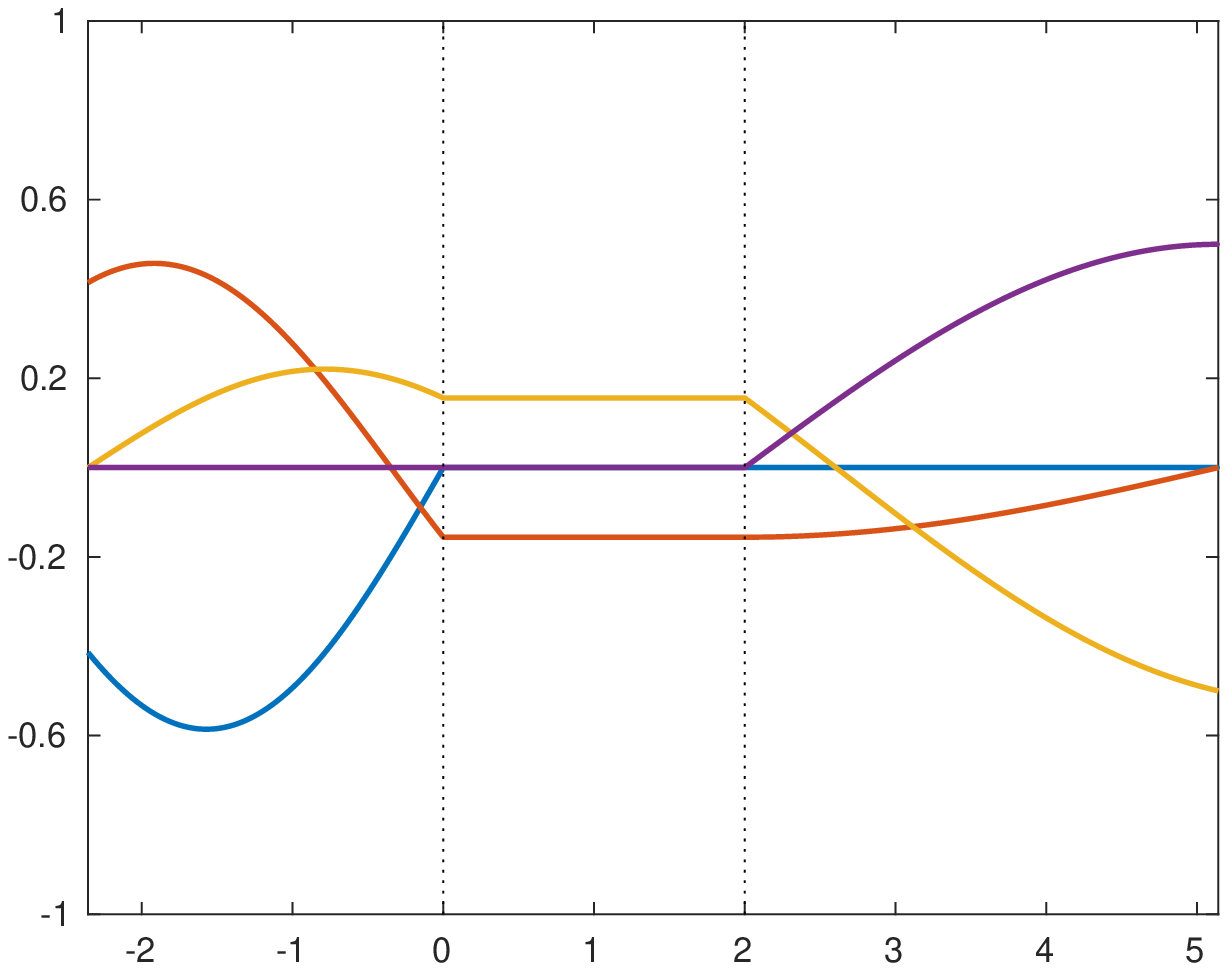}}
\caption{A profile consisting of two circular arcs connected by a straight line defined in \cref{ex:2} (left column), together with the $C^1$ GTB-splines used in the representation (middle column) and their first derivatives (right column). Knot positions are visualized by vertical dotted lines.}
\label{fig:ex2}
\end{figure}

\begin{remark} The nice behavior of the control polygon in \cref{ex:2} (and corresponding \cref{fig:ex2}) is not a coincidence. The control polygon is easy to construct for any $C^1$ GTB-spline curve whose components belong to a GT-spline space consisting of ECT-spaces of dimension 3 and 2 in an alternating sequence. The control points are given by: (a) the two end points of the curve, and (b) the ordered intersections of the two end point tangent lines with the tangent lines corresponding to the linear segments.
\end{remark}

%% file: sections/conclusion.tex
\section{Conclusion} \label{sec:conclusion}
GT-spline spaces are smooth function spaces where the pieces are drawn from different ET-spaces of possibly different dimensions. Under quite mild assumptions, they offer the possibility of exploiting the wide flexibility of ET-spaces while retaining the nice properties of classical polynomial spline spaces, including a B-spline-like basis, the so-called GTB-splines.

Besides their common application in constrained interpolation/approximation and computer-aided geometric design --- the richness of ET-spaces offers a huge universe of shapes for modeling --- GTB-splines can be a powerful tool for numerical simulation as well. Since a relevant class of ET-spaces can be specified as the nullspaces of linear differential operators, it is clear that GT-splines and GTB-splines can provide an appealing problem-dependent alternative to classical polynomial splines for the numerical treatment of differential and integral problems. Heretofore, this great potential has been thwarted by the lack of efficient and reliable evaluation procedures for GTB-splines, even in the simpler case where ET-spaces of the same dimension are glued together.

We have presented an efficient and robust algorithm for evaluation of GTB-splines whenever they exist. The algorithm proceeds by incrementally increasing the smoothness starting from the space of piecewise discontinuous functions obtained by collecting the various ET-spaces. It requires as input the local Bernstein-like bases and produces as output the entire set of GTB-splines that span the considered GT-spline space. The algorithm recursively constructs the nullspace of a suitable matrix in a numerically stable way without solving a linear system. In contrast with the current available methods for evaluation of GTB-splines, the proposed strategy does not require any (numerical) integration. Indeed, integration can be avoided also to produce the starting Bernstein-like bases as they can be obtained by solving suitable local Hermite interpolation problems and this can be done in a pre-processing step. The provided algorithm is a \Chef extension of the procedure recently developed and analyzed in \cite{Speleers:2018,Toshniwal:2017a,Toshniwal:2018b} for multi-degree polynomial splines.

The considered ET-spaces, defined on the bounded and closed intervals identified by the breakpoints, are represented in terms of weights, possibly constrained by some admissibility conditions. It should be noted, however, that this is merely for the sake of presentation, so as to have a framework where GTB-splines exist. The proposed algorithm does not require any weights and always produces the GTB-spline basis whenever it exists. Therefore, the end-user can completely ignore this representation in terms of weights when just interested in the computation of GTB-splines.
{\Rd Actually, the algorithm is applicable to any kind of spline space that is equipped with a B-spline-like basis (in the sense of \cref{rmk:abstract-basis}), also beyond our \Chef setting. In this perspective, it will be interesting to explore the multi-degree framework in the context of variable degree polynomial splines \cite{Bosner:2013,Costantini:2000,Mazure:2001}.}

The majority of works on \Chef splines deals with spline spaces obtained by gluing together ET-spaces of the same dimension. In this context, mainly motivated by computer-aided geometric design as application, the concept of geometric continuity is often considered instead of classical continuity. Geometric continuity offers additional shape parameters for design. However, this flexibility comes at a price of increased complexity and can be of practical interest only when equipped with proper, preferably automatic, strategy for parameter selection. In this paper, we have deliberately confined ourselves to classical continuity with the aim of promoting the use of GT-splines in the wider context of numerical simulation and more precisely in isogeometric methods, where the choice of the ET-spaces has to be driven by the character of the problem under consideration. Nevertheless, the presented procedure has the potential to construct an efficient evaluation algorithm for geometrically continuous \Chef splines as well. Future research efforts will also focus on multivariate extensions of the algorithmic evaluation approach; a particularly interesting topic in this direction is the construction of a B-spline-like basis for GT-splines on T-meshes.

%% file: ex_article.bbl
\begin{thebibliography}{10}

\bibitem{Aimi:2017}
{\sc A.~Aimi, M.~Diligenti, M.~L. Sampoli, and A.~Sestini}, {\em Non-polynomial
  spline alternatives in isogeometric symmetric {G}alerkin {BEM}}, Appl. Numer.
  Math., 116 (2017), pp.~10--23.

\bibitem{Beccari:2017}
{\sc C.~V. Beccari, G.~Casciola, and S.~Morigi}, {\em On multi-degree splines},
  Comput. Aided Geom. Design, 58 (2017), pp.~8--23.

\bibitem{BisterP:1997}
{\sc D.~Bister and H.~Prautzsch}, {\em A new approach to {T}chebycheffian
  {B}-splines}, in Curves and Surfaces with Applications in CAGD,
  A.~Le~M\'ehaut\'e, C.~Rabut, and L.~L. Schumaker, eds., Vanderbilt Univ.
  Press, 1997, pp.~387--394.

\bibitem{Deboor:1978}
{\sc C.~d. Boor}, {\em A Practical Guide to Splines, Revised Edition},
  Springer--Verlag, 2001.

\bibitem{Borden:2011}
{\sc M.~J. Borden, M.~A. Scott, J.~A. Evans, and T.~J.~R. Hughes}, {\em
  Isogeometric finite element data structures based on {B}{\'e}zier extraction
  of {NURBS}}, Int. J. Numer. Meth. Eng., 87 (2011), pp.~15--47.

\bibitem{Bosner:2013}
{\sc T.~Bosner and M.~Rogina}, {\em Variable degree polynomial splines are
  {C}hebyshev splines}, Adv. Comput. Math., 38 (2013), pp.~383--400.

\bibitem{BraccoLMRS:2016-gb}
{\sc C.~Bracco, T.~Lyche, C.~Manni, F.~Roman, and H.~Speleers}, {\em
  Generalized spline spaces over {T}-meshes: Dimension formula and locally
  refined generalized {B}-splines}, Appl. Math. Comput., 272 (2016),
  pp.~187--198.

\bibitem{BraccoLMRS:2016}
{\sc C.~Bracco, T.~Lyche, C.~Manni, F.~Roman, and H.~Speleers}, {\em On the
  dimension of {T}chebycheffian spline spaces over planar {T}-meshes}, Comput.
  Aided Geom. Design, 45 (2016), pp.~151--173.

\bibitem{BraccoLMRS:2019}
{\sc C.~Bracco, T.~Lyche, C.~Manni, and H.~Speleers}, {\em {T}chebycheffian
  spline spaces over planar {T}-meshes: Dimension bounds and dimension
  instabilities}, J. Comput. Appl. Math., 349 (2019), pp.~265--278.

\bibitem{Buchwald:2003}
{\sc B.~Buchwald and G.~M{\"u}hlbach}, {\em Construction of {B}-splines for
  generalized spline spaces generated from local {ECT}-systems}, J. Comput.
  Appl. Math., 159 (2003), pp.~249--267.

\bibitem{Carnicer:2003}
{\sc J.~M. Carnicer, E.~Mainar, and J.~M. Pe{\~n}a}, {\em Critical length for
  design purposes and extended {C}hebyshev spaces}, Constr. Approx., 20 (2003),
  pp.~55--71.

\bibitem{Coppel:1971}
{\sc W.~A. Coppel}, {\em Disconjugacy}, Springer--Verlag, 1971.

\bibitem{Costantini:2000}
{\sc P.~Costantini}, {\em Curve and surface construction using variable degree
  polynomial splines}, Comput. Aided Geom. Design, 17 (2000), pp.~419--446.

\bibitem{Costantini:2005}
{\sc P.~Costantini, T.~Lyche, and C.~Manni}, {\em On a class of weak
  {T}chebycheff systems}, Numer. Math., 101 (2005), pp.~333--354.

\bibitem{DynR:1988}
{\sc N.~Dyn and A.~Ron}, {\em Recurrence relation for {T}chebycheffian
  {B}-splines}, J. Anal. Math., 51 (1988), pp.~118--138.

\bibitem{Karlin:1968}
{\sc S.~Karlin}, {\em Total Positivity}, Stanford Univ. Press, 1968.

\bibitem{KarlinS:1966}
{\sc S.~Karlin and W.~J. Studden}, {\em Tchebycheff Systems: With Applications
  in Analysis and Statistics}, Interscience Publishers, 1966.

\bibitem{Lyche:1985}
{\sc T.~Lyche}, {\em A recurrence relation for {C}hebyshevian {B}-splines},
  Constr. Approx., 1 (1985), pp.~155--173.

\bibitem{Lyche:2018}
{\sc T.~Lyche, C.~Manni, and H.~Speleers}, {\em Foundations of spline theory:
  {B}-splines, spline approximation, and hierarchical refinement}, in Splines
  and PDEs: From Approximation Theory to Numerical Linear Algebra, T.~Lyche
  et~al., eds., vol.~2219 of Lecture Notes in Mathematics, Springer
  International Publishing AG, 2018, pp.~1--76.

\bibitem{Lyche:2019}
{\sc T.~Lyche, C.~Manni, and H.~Speleers}, {\em Tchebycheffian {B}-splines
  revisited: An introductory exposition}, in Advanced Methods for Geometric
  Modeling and Numerical Simulation, C.~Giannelli and H.~Speleers, eds.,
  vol.~35 of Springer INdAM Series, Springer International Publishing AG, 2019,
  pp.~179--216.

\bibitem{LycheS:2000}
{\sc T.~Lyche and L.~L. Schumaker}, {\em A multiresolution tensor spline method
  for fitting functions on the sphere}, SIAM J. Sci. Comput., 22 (2000),
  pp.~724--746.

\bibitem{Mainar:2001}
{\sc E.~Mainar, J.~M. Pe\~na, and J.~S\'anchez-Reyes}, {\em Shape preserving
  alternatives to the rational {B}\'ezier model}, Comput. Aided Geom. Design,
  18 (2001), pp.~37--60.

\bibitem{ManniPS:2011a}
{\sc C.~Manni, F.~Pelosi, and M.~L. Sampoli}, {\em Generalized {B}-splines as a
  tool in isogeometric analysis}, Comput. Methods Appl. Mech. Eng., 200 (2011),
  pp.~867--881.

\bibitem{ManniPS:2011b}
{\sc C.~Manni, F.~Pelosi, and M.~L. Sampoli}, {\em Isogeometric analysis in
  advection-diffusion problems: Tension splines approximation}, J. Comput.
  Appl. Math., 236 (2011), pp.~511--528.

\bibitem{ManniRS:2015}
{\sc C.~Manni, A.~Reali, and H.~Speleers}, {\em Isogeometric collocation
  methods with generalized {B}-splines}, Comput. Math. Appl., 70 (2015),
  pp.~1659--1675.

\bibitem{ManniRS:2017}
{\sc C.~Manni, F.~Roman, and H.~Speleers}, {\em Generalized {B}-splines in
  isogeometric analysis}, in Approximation Theory XV: San Antonio 2016, G.~E.
  Fasshauer and L.~L. Schumaker, eds., vol.~201 of Springer Proceedings in
  Mathematics \& Statistics, Springer International Publishing AG, 2017,
  pp.~239--267.

\bibitem{Mazure:2001}
{\sc M.-L. Mazure}, {\em Quasi-{C}hebychev splines with connection matrices:
  Application to variable degree polynomial splines}, Comput. Aided Geom.
  Design, 18 (2001), pp.~287--298.

\bibitem{Mazure:2007}
{\sc M.-L. Mazure}, {\em Extended {C}hebyshev piecewise spaces characterised
  via weight functions}, J. Approx. Theory, 145 (2007), pp.~33--54.

\bibitem{Mazure:2011b}
{\sc M.-L. Mazure}, {\em Finding all systems of weight functions associated
  with a given extended {C}hebyshev space}, J. Approx. Theory, 163 (2011),
  pp.~363--376.

\bibitem{Mazure:2011a}
{\sc M.-L. Mazure}, {\em How to build all {C}hebyshevian spline spaces good for
  geometric design?}, Numer. Math., 119 (2011), pp.~517--556.

\bibitem{Mazure:2018}
{\sc M.-L. Mazure}, {\em Constructing totally positive piecewise {C}hebyshevian
  {B}-spline bases}, J. Comput. Appl. Math., 342 (2018), pp.~550--586.

\bibitem{Muhlbach:2006}
{\sc G.~M{\"u}hlbach}, {\em {ECT}-{B}-splines defined by generalized divided
  differences}, J. Comput. Appl. Math., 187 (2006), pp.~96--122.

\bibitem{Nurnberger:1984}
{\sc G.~N\"urnberger, L.~L. Schumaker, M.~Sommer, and H.~Strauss}, {\em
  Generalized {C}hebyshevian splines}, SIAM J. Math. Anal., 15 (1984),
  pp.~790--804.

\bibitem{Schumaker:2007}
{\sc L.~L. Schumaker}, {\em Spline Functions: Basic Theory, 3rd Edition},
  Cambridge Univ. Press, 2007.

\bibitem{Schweikert:1966}
{\sc D.~G. Schweikert}, {\em An interpolation curve using a spline in tension},
  J. Math. Phys., 45 (1966), pp.~312--317.

\bibitem{Scott:2011}
{\sc M.~A. Scott, M.~J. Borden, C.~V. Verhoosel, T.~W. Sederberg, and T.~J.~R.
  Hughes}, {\em Isogeometric finite element data structures based on
  {B}{\'e}zier extraction of {T}-splines}, Int. J. Numer. Meth. Eng., 88
  (2011), pp.~126--156.

\bibitem{Sederberg:2003a}
{\sc T.~W. Sederberg, J.~Zheng, and X.~Song}, {\em Knot intervals and
  multi-degree splines}, Comput. Aided Geom. Design, 20 (2003), pp.~455--468.

\bibitem{ShenW:2010a}
{\sc W.~Shen and G.~Wang}, {\em A basis of multi-degree splines}, Comput. Aided
  Geom. Design, 27 (2010), pp.~23--35.

\bibitem{ShenW:2010b}
{\sc W.~Shen and G.~Wang}, {\em Changeable degree spline basis functions}, J.
  Comput. Appl. Math., 234 (2010), pp.~2516--2529.

\bibitem{Speleers:2018}
{\sc H.~Speleers}, {\em Algorithm 999: Computation of multi-degree
  {B}-splines}, ACM Trans. Math. Softw., 45 (2019), Article 43.

\bibitem{Toshniwal:2017a}
{\sc D.~Toshniwal, H.~Speleers, R.~R. Hiemstra, and T.~J.~R. Hughes}, {\em
  Multi-degree smooth polar splines: A framework for geometric modeling and
  isogeometric analysis}, Comput. Methods Appl. Mech. Eng., 316 (2017),
  pp.~1005--1061.

\bibitem{Toshniwal:2018b}
{\sc D.~Toshniwal, H.~Speleers, R.~R. Hiemstra, C.~Manni, and T.~J.~R. Hughes},
  {\em Multi-degree {B}-splines: Algorithmic computation and properties},
  Comput. Aided Geom. Design, 76 (2020), Article 101792.

\bibitem{WangF:2008}
{\sc G.~Wang and M.~Fang}, {\em Unified and extended form of three types of
  splines}, J. Comput. Appl. Math., 216 (2008), pp.~498--508.

\end{thebibliography}
